\documentclass[a4paper]{article}
\usepackage{fullpage}
\usepackage{listings}
\usepackage{latexsym, amsfonts, amssymb, amsmath, amsthm, mathrsfs, mathtools, setspace, graphics, graphicx, bbm, float,bigints, color}
\usepackage{enumerate}
\usepackage{caption}
\usepackage{indentfirst}
\usepackage{framed}
\usepackage{yhmath}
\RequirePackage{tikz}
\usetikzlibrary{calc,trees,positioning,arrows,chains,shapes.geometric,%
    decorations.pathreplacing,decorations.pathmorphing,shapes,%
    matrix,shapes.symbols}

\allowdisplaybreaks

\usepackage{geometry}
\title{Characterization of probability distribution convergence in Wasserstein distance by $L^{p}$-quantization error function}
\usepackage{hyperref}
\newcommand{\footremember}[2]{
   \footnote{#2}
    \newcounter{#1}
    \setcounter{#1}{\value{footnote}}
}

\author{
     Yating LIU\footremember{t}{Sorbonne Universit\'e, Laboratoire de Probabilit\'es, Statistique et Mod\'elisation, UMR 8001, case 188, 4, pl. Jussieu, F-75252 Paris Cedex 5, France. E-mail: \texttt{yating.liu@sorbonne-universite.fr}}
     \and Gilles PAG\`ES\footremember{a}{Sorbonne Universit\'e, Laboratoire de Probabilit\'es, Statistique et Mod\'elisation, UMR 8001, case 188, 4, pl. Jussieu, F-75252 Paris Cedex 5, France. E-mail: \texttt{gilles.pages@sorbonne-universite.fr}}
}
\begin{document}
\maketitle

\renewcommand{\thefootnote}{(\arabic{footnote})}

\theoremstyle{plain}
\newtheorem{lem}{Lemma}[section]
\newtheorem{prop}{Proposition}[section]
\newtheorem{cor}{Corollary}[section]

\theoremstyle{definition}
\newtheorem{defn}{Definition}[section]
\newtheorem{thm}{Theorem}[section]

\theoremstyle{remark}
\newtheorem*{rem}{Remark}

\begin{abstract}
We establish  conditions to  characterize probability measures by their $L^{p}$-quantization error functions in both $\mathbb{R}^{d}$ and Hilbert settings. This characterization is two-fold:  static (identity of two distributions) and dynamic (convergence  for the $L^p$-Wasserstein distance). We first propose a criterion on the qantization level $N$, valid for any norm on $\mathbb{R}^{d}$  and  any order $p$ based on  a geometrical approach involving the Vorono\"i diagram. Then, we prove that in the $L^2$-case  on a (separable) Hilbert space, the condition on the level $N$  can be reduced to $N=2$, which is optimal. More quantization based characterization cases on dimension 1 and a discussion of the completeness of a distance defined by the quantization error function can be found in the end of this paper.
\end{abstract}

\emph{Keywords:}  Probability distribution characterization; Vector quantization; Vorono\"i diagram; Wasserstein convergence.

\section{Introduction}

Vector quantization was originally developed as an optimal discretization method for signal transmission and compression by the Bell laboratories in the 1950s. Many seminal and historical contributions on vector quantization and its connections with information theory were gathered and published later in~\cite{IEEE1982}. In the unsupervised learning area, vector quantization has a close connection with the automatic classification (clustering)  through the $k$-means algorithm. 
More recently, in the 1990s, it  became an efficient tool in numerical probability to compute regular and  conditional expectations (see~\cite{pages1998space},~\cite{bally2000quantization} and~\cite{pages2003optimal}) with in view the pricing of derivative products.
Thus, a quantization based numerical schemes have been developed for  American option pricing (see~\cite{bally2002quantization}), and for the simulation of Backward Stochastic Differential Equation or  nonlinear filtering (see~\cite{pages2018improved}). For a first  mathematically  rigorous  monograph of various aspects of vector quantization theory, we refer to \cite{graf2000foundations} (and the references therein). For more  engineering applications  to signal compression see e.g.~\cite{gersho2012vector} among an extensive literature.

In all these applications, either of probabilistic or statistical nature, vector quantization is used to produce a kind of skeleton of a probability distribution.
To be more precise, let $(\Omega, \mathcal{A}, \mathbb{P})$ denote a probability space and let $X$ be a random variable defined on $(\Omega, \mathcal{A}, \mathbb{P})$ and valued in $(E, |\cdot|_{E})$, where $E$ is $\mathbb{R}^{d}$ or a separable Hilbert space $H$ and $|\cdot|_{E}$ denotes respectively the norm on $\mathbb{R}^{d}$ or the norm on $H$ induced by the inner product $(\cdot\,|\,\cdot)_{H}$. Let $\mu$ denote the probability distribution of $X$, denoted by $\mathbb{P}_{X}=\mu$ or $\text{Law}(X)=\mu$ and assume that $\mu$ has a finite $p$-th moment. 
The quantization grid (also called \textit{codebook} in signal compression or \textit{cluster center} in machine learning theory) is a finite set of points in $E$, denoted by $\Gamma=\{x_{1}, ..., x_{N}\}\subset E$. Let us define the distance between a point $\xi$ and a set $A$ in $E$ by $d(\xi, A)=\min_{a\in A}\left|\xi-a\right|_{E}$. The $L^{p}$-mean quantization error of $\Gamma$, defined by $e_{p}(\mu, \Gamma)\coloneqq\left\Vert d(X, \Gamma)\right\Vert_{p}=\Big[\int_{E}\min_{a\in\Gamma}\left|\xi-a\right|_{E}^{p}\mu(d\xi)\Big]^{\frac{1}{p}}$, is used to describe the accuracy level of representing the probability measure $\mu$ by $\Gamma$. Let $N\ge 1$. A  quantization grid $\Gamma^{*, (N)}$ satisfying 
\begin{equation}\label{errorgrid}
e_{p}(\mu, \Gamma^{*, (N)})=\inf_{\substack{\Gamma\subset E,\\\text{card}(\Gamma)\leq N}}\Big[\mathbb{E}\;d(X,\Gamma)^{p}\Big]^{\frac{1}{p}}=\inf_{\substack{\Gamma\subset E,\\\text{card}(\Gamma)\leq N}}\Big[\int_{E}\min_{a\in\Gamma}\left|\xi-a\right|_{E}^{p}\mu(d\xi)\Big]^{\frac{1}{p}}
\end{equation}
is called an \textit{$L^{p}$-optimal quantization grid} (or \emph{optimal grid} in short) at {\em level} $N$. We refer to~\cite{graf2000foundations}[Theorem 4.12] for the existence of such an optimal grids on $\mathbb{R}^{d}$ and to~\cite{luschgy2002functional}[Proposition 2.1] or~\cite{cuesta1988strong} on  (separable) Hilbert spaces. 
There is usually no closed form for optimal grids, however, in the quadratic case ($p=2$), it can be computed by the stochastic optimization methods such as the CLVQ algorithm or the randomized Lloyd algorithm (see~\cite{pages2015introduction}[Section 3],~\cite{kieffer1982exponential} and~\cite{pages2016pointwise}).

Optimal grids $\Gamma^{*, (N)}$ ``carries'' the 
information of the initial measure. For example, let $\mu\!\in\mathcal{P}_{p+\varepsilon}(\mathbb{R}^{d})$ for some $\varepsilon>0$, where $\mathcal{P}_{p}(E)\coloneqq\{\mu \text{ probability distribution on } E$ $ \,\text{s.t.}\, \int_{E}\left|\xi\right|_{E}^{p}\mu(d\xi)<+\infty\}$.  Let $\mu=h\cdot\lambda_{d}$ be an absolutely continuous distribution ($\lambda_d$  denotes  Lebesgue measure). 
If for every level $N\geq1$, $\Gamma^{*, (N)}$ is an optimal quantization grid of $\mu$ at level $N$, then 
\begin{equation}\label{example}
\frac{1}{N}\sum_{x\in\Gamma^{*, (N)}}\delta_{x}\xRightarrow{\;(\mathbb{R}^{d})\;}\widetilde{\mu}=\frac{h^{d/(d+p)}(\xi)}{\int h^{d/(d+p)}d\lambda_{d}}\lambda_{d}(d\xi), \;\;\text{as}\;\; N\rightarrow+\infty,
\end{equation}
where, for a Polish space $S$, $\xRightarrow{(S)}$ denotes the weak convergence of probability measures on $S$. We refer to~\cite{graf2000foundations}[Theorem 7.5] for a proof of this result. 
This weak convergence~(\ref{example}) emphasizes that,
an absolutely continuous probability measure $\mu$ is entirely characterized by the sequence of $L^{p}$-optimal quantization grids $\Gamma^{*,(N)}$ at levels $N$, $N\geq1$.

We consider now the $L^{p}$-mean quantization error function as follows.

\begin{defn}[Quantization error function]\label{defquanerrf}
Let $\mu\!\in\mathcal{P}_{p}(\mathbb{R}^{d})$, $p\in[1, +\infty)$. The 
$L^{p}$-mean quantization error function
of $\mu$ at level $N$, denoted by $e_{N,p}(\mu, \cdot)$, is defined by:
\begin{equation}\label{lpquantierror}
\begin{array}{cccc}
e_{N,p}(\mu,\cdot):&(\mathbb{R}^{d})^{N}&\longrightarrow & \mathbb{R}_{+}\\
&x=(x_{1}, \dots, x_{N}) & \longmapsto& \displaystyle e_{N,p}(\mu, x)=\Big[\int_{\mathbb{R}^{d}}\min_{1\leq i\leq N}\left|\xi-x_{i}\right|^{p}\mu(d\xi)\Big]^{\frac{1}{p}}.
\end{array}
\end{equation}
\end{defn}

\vspace{-0.2cm}
\noindent The definition of $e_{N,p}(\mu,\cdot)$
obviously depends on the associated norm on $\mathbb{R}^{d}$ and the variable of $e_{N,p}(\mu, \cdot)$ is a priori an $N$-tuple in $(\mathbb{R}^{d})^{N}$. However, for a finite grid $\Gamma\subset\mathbb{R}^{d}$, if the level $N\geq \text{card}(\Gamma)$, then for any $N$-tuple $x^{\Gamma}=(x_{1}^{\Gamma}, \dots, x_{N}^{\Gamma})\in(\mathbb{R}^{d})^{N}$ such that $\Gamma=\{x_{1}^{\Gamma}, \dots, x_{N}^{\Gamma}\}$,  we have $e_{p}(\mu, \Gamma)=e_{N,p}(\mu, x^{\Gamma})$.
For example, $e_p\big(\mu, \{x_1,x_2\}\big) = e_{2,p}\big(\mu, (x_1,x_2)\big) = e_{3,p}\big(\mu, (x_1,x_1,x_2)\big)$, etc.
Note that $e_{N, p}$ is a symmetric function 
on $(\mathbb{R}^{d})^{N}$ and that, owing to the above definition, 
\begin{equation}\label{argmin}
\inf_{\Gamma\subset\mathbb{R}^{d}, \text{card}(\Gamma)\leq N}e_{p}(\mu, \Gamma)=\inf_{x\in(\mathbb{R}^{d})^{N}}e_{N, p}(\mu, x).
\end{equation}
Therefore, throughout  this paper, with a slight abuse of notation,  we will also denote the $L^{p}$-quantization error at level $N$ for a grid $\Gamma$ of size at most $N$ by $e_{N,p}(\mu, \Gamma)$. 

The equality (\ref{argmin}) directly shows that the optimal grids are characterized by the $L^{p}$-mean quantization error functions.  Next, we show that the quantization error function $e_{N, p}(\mu, \cdot)$ is entirely characterized by the probability distribution $\mu$. 

Notice that for any $\mu\!\in\mathcal{P}_{p}(\mathbb{R}^{d})$, the function $e_{N,p}(\mu, \cdot)$ defined in (\ref{lpquantierror}) is 1-Lipschitz continuous for every $N\geq1$ since for any $x=(x_{1}, \dots, x_{N}),y=(y_{1}, \dots, y_{N})\in(\mathbb{R}^{d})^{N}$,
\begin{align}\label{1lip}
&\left|e_{N,p}(\mu, x)-e_{N,p}(\mu, y)\right|=\left|\Big[\int_{\mathbb{R}^{d}}\min_{1\leq i\leq N}\left|\xi-x_{i}\right|^{p}\mu(d\xi)\Big]^{\frac{1}{p}}-\Big[\int_{\mathbb{R}^{d}}\min_{1\leq j\leq N}\left|\xi-y_{j}\right|^{p}\mu(d\xi)\Big]^{\frac{1}{p}}\right|\nonumber\\
&\hspace{1cm}\leq\Big[\int_{\mathbb{R}^{d}}\Big|\min_{1\leq i\leq N}\left|\xi-x_{i}\right|-\min_{1\leq j\leq N}\left|\xi-y_{j}\right|\Big|^{p}\mu(d\xi)\Big]^{\frac{1}{p}}\hspace{0.5cm}\text{\small(by the Minkowski inequality)}\nonumber\\
&\hspace{1cm}\leq\Big[\int_{\mathbb{R}^{d}}\max_{1\leq i\leq N}\left|x_{i}-y_{i}\right|^{p}\mu(d\xi)\Big]^{\frac{1}{p}}=\max_{1\leq i\leq N}\left|x_{i}-y_{i}\right|.
\end{align}

We recall now the definition of the $L^p$-Wasserstein distance.

\begin{defn}[$L^p$-Wasserstein distance]\label{defwas}
Let $(S,d)$ be a Polish space and $\mathcal{S}=\text{Bor}(S,d)$ be its Borel $\sigma$-field. For $p\in[1, +\infty)$, let $\mathcal{P}_{p}(S)$ denote the set of probability measures on $(S,\mathcal{S})$ with a finite $p^{th}$-moment. 
The $L^p$-Wasserstein  distance  $\mathcal{W}_{p}(\mu,\nu)$ between $\mu,\, \nu\!\in\mathcal{P}_{p}(S)$, denoted by $\mathcal{W}_{p}(\mu,\nu)$, is defined by
\begin{align}\label{defwas2}
\mathcal{W}_{p}(\mu,\nu)&=\Big{(}\inf_{\pi\!\in\Pi(\mu,\nu)}\int_{S\times S}d(x,y)^{p}\pi(dx,dy)\Big{)}^{\frac{1}{p}}\nonumber\\
&=\inf\Big{\{}\Big{[}\mathbb{E}\;d(X,Y)^{p}\Big{]}^{\frac{1}{p}},\;  X,Y:(\Omega,\mathcal{A},\mathbb{P})\rightarrow(\text{S},\mathcal{S})  \;\text{with} \;\mathbb{P}_{X}=\mu, \mathbb{P}_{Y}=\nu\;\Big{\}},
\end{align}
\noindent where in the first line of (\ref{defwas2}), $\Pi(\mu,\nu)$ denotes the set of all probability measures on $(S^{2}, \mathcal{S}^{\otimes2})$ with respective marginals $\mu$ and $\nu$.
\end{defn}

If we consider  $e_{N,p}(\mu, x)$ as a function of $\mu\!\in\mathcal{P}_{p}(\mathbb{R}^{d})$, then $e_{N,p}$ is also $1$-Lipschitz in $\mu$. In fact, let $X,Y$ be two random variables with probability distributions $\mathbb{P}_{X}=\mu$ and $\mathbb{P}_{Y}=\nu$. For every $N$-tuple $x=(x_{1},\dots,x_{N})\in(\mathbb{R}^{d})^{N}$, we have
\begin{align}\label{wast}
&\big|e_{N,p}(\mu, x)-e_{N,p}(\nu, x)\big|=\left| \;\;\Big\Vert\min_{i=1,\dots,N}\left|X-x_{i}\right|\Big\Vert_{p}-\Big\Vert\min_{i=1,\dots,N}\left|Y-x_{i}\right|\Big\Vert_{p}\;\; \right|\nonumber\\
&\hspace{1cm}\leq\Big\Vert\min_{i=1,\dots,N}\left|X-x_{i}\right|-\min_{i=1,\dots,N}\left|Y-x_{i}\right|\Big\Vert_{p} \text{\small(by the Minkowski inequality)}\nonumber\\
&\hspace{1cm}\leq\Big\Vert\max_{i=1,\dots,N}\left|\;\;\left|X-x_{i}\right|-\left|Y-x_{i}\right|\;\;\right|\Big\Vert_{p}\leq\left\Vert X-Y\right\Vert_{p}.
\end{align}

As this inequality holds for every couple $(X, Y)$ of random variables with marginal distributions $\mu$ and $\nu$, it follows that for  every level $N\geq1$,
\begin{equation}\label{control}
\left\Vert e_{N, p}(\mu, \cdot)-e_{N,p}(\nu, \cdot)\right\Vert_{\sup}\coloneqq\sup_{x\in(\mathbb{R}^{d})^{N}}\left|e_{N,p}(\mu, x)-e_{N,p}(\nu, x)\right|\leq\mathcal{W}_{p}(\mu,\nu).
\end{equation}
Hence, if $(\mu_{n})_{n\geq1}$ is a sequence in $\mathcal{P}_{p}(\mathbb{R}^{d})$ converging for the $\mathcal{W}_{p}$-distance to $\mu_{\infty}\!\in\mathcal{P}_{p}(\mathbb{R}^{d})$, then
\begin{equation}\label{06}
\left\Vert e_{N,p}(\mu_{n},\cdot)-e_{N,p}(\mu_{\infty},\cdot)\right\Vert_{\sup}\leq\mathcal{W}_{p}(\mu_{n},\mu_{\infty})\xrightarrow{n\rightarrow+\infty}0.
\end{equation}

Definition~\ref{defquanerrf}, and the inequalities (\ref{1lip}), (\ref{wast}), (\ref{control}), (\ref{06}) can be directly extended to any separable Hilbert space $H$. Inequalities (\ref{control}) and (\ref{06}) show that for every $N\ge 1$, and $p\!\in [1, +\infty)$, the quantization error function $e_{N, p}(\mu, \cdot)$ is characterized by the probability distribution $\mu$. Hence, the characterization relations between a probability measure $\mu$, its $L^{p}$-quantization error function and its optimal grids can be synthesized by the following scheme:

\tikzstyle{mybox} = [draw=black, fill=white, thick,
    rectangle, rounded corners,  inner sep=5pt, inner ysep=5pt]
\tikzstyle{fancytitle}=[draw=black, fill=white, rectangle, rounded corners, thick, text=black]
\tikzstyle{arrow} = [->,>=stealth]

\begin{figure}[H]
\centering

\begin{tikzpicture}
\node (probameasure) [mybox]
{Probability measure $\mu$
};

\node (quanerrorfunc) [mybox, below left = 2.2cm and -0.9cm of probameasure]
{\begin{tabular}{cc}
Quantization error \\
function $e_{N, p}(\mu, \cdot)$
\end{tabular}};

\node (optimalgird) [mybox, right =2cm of quanerrorfunc]
{\begin{tabular}{cc}
Optimal grid \\$\Gamma^{*, (N)}$
\end{tabular}};

\draw[->,>=stealth, xshift=4cm] (probameasure) to [bend right=20] node[midway, left] {See (\ref{control}) and (\ref{06})$\;$}(quanerrorfunc) ;
\draw[->,densely dashed, >=stealth, , xshift=-2cm] (quanerrorfunc) to [bend right=10] node [midway, right] {\large $\;$?}(probameasure)  ;
\draw[->,>=stealth, , xshift=-2cm] (quanerrorfunc) to [bend right =10] node [midway, below] {\begin{tabular}{cc}argmin\\ see (\ref{argmin})\end{tabular}} (optimalgird);
\draw[->,>=stealth, , xshift=-2cm] (optimalgird) to [bend right=20] node [midway, right] {\begin{tabular}{cc} If $\mu\!\in\mathcal{P}_{p+\varepsilon}(\mathbb{R}^{d})$, $\mu\ll\lambda_{d}$ \\(absolutely  continuous) and  \\$\;$ if we know the optimal grid \\$\;\;$ for every level $N$, see (\ref{example}).\end{tabular}}(probameasure);
\end{tikzpicture}
\end{figure}

\vskip-0.25cm The characterization of a probability measure $\mu$ by its $L^{p}$-optimal quantization grids suggests to consider the ``reverse'' questions of (\ref{control}) and (\ref{06}): \emph{When is a probability measure $\mu\!\in\mathcal{P}_{p}(\mathbb{R}^{d})$ characterized by its $L^{p}$-quantization error function $e_{N,p}(\mu, \cdot)$? And if so, does the convergence in an appropriate sense of the $L^{p}$-quantization error functions characterizes the convergence of their probability distributions for the $\mathcal{W}_{p}$-distance?}

These questions can be formalized as follows (the first one in a slightly extended sense):
\begin{framed}

\vspace{-0.5cm}
\begin{itemize}
\item \textbf{Question 1 - Static characterization}: \\
If for $\mu, \nu\!\!\in\mathcal{P}_{p}(\mathbb{R}^{d})$,  $e_{N, p}(\mu, \cdot)=e_{N, p}(\nu, \cdot)+ C$ for some real constant $C$, then do we have $\mu=\nu$ (and $C=0$)?
\item \textbf{Question 2 - Characterization of $\mathcal{W}_{p}$-convergence}: \\
If for  $\mu_{n}, n\geq1$, $\mu_{\infty}\!\!\in\mathcal{P}_{p}(\mathbb{R}^{d})$, $e_{N, p}(\mu_{n}, \cdot)$ converges pointwise to $e_{N, p}(\mu_{\infty}, \cdot)$, then do we have  $\mathcal{W}_{p}(\mu_{n},\mu_{\infty})\xrightarrow{n\rightarrow+\infty}0$?
\end{itemize}
\vspace{-0.2cm}
\end{framed}

For any $N_{1}, N_{2}\!\in\mathbb{N}^{*}$ with $N_{1}\leq N_{2}$, it is clear that
$e_{N_{2},p}(\mu, \cdot)=e_{N_{2},p}(\nu, \cdot)$ \big(resp. $e_{N_{2},p}(\mu_{n}, \cdot)\xrightarrow{n\rightarrow+\infty}e_{N_{2},p}(\mu_{\infty}, \cdot)$\big) implies 
$e_{N_{1},p}(\mu, \cdot)=e_{N_{1},p}(\nu, \cdot)$ \big(resp. $e_{N_{1},p}(\mu_{n}, \cdot)\xrightarrow{n\rightarrow+\infty}e_{N_{1},p}(\mu_{\infty}, \cdot)$\big). Hence, beyond these two above questions, we  need to determine an as low as possible level $N$ for which both answers are positive.  
For this purpose, we define
\begin{align}\label{defndq}
N_{d, p, \left|\cdot\right|}\coloneqq\min\{N\in&\,\mathbb{N}^{*} \text{ such that  answers to {\em Questions 1} and {\em 2} for $e_{N, p}$  are positive}\}.
\end{align}

The paper is organized as follows. We first recall in Section~\ref{wasserstein} some properties of the Wasserstein distance $\mathcal{W}_{p}$. Then in Section~\ref{genernalcase}, we begin to analyze the problem of probability distribution characterization in a general finite dimensional framework by considering any dimension $d$, any order $p$ and any norm on $\mathbb{R}^{d}$. We show that a positive answer to \textit{Question 1 and 2} follows from the existence of a bounded open Vorono\"i cell in a Vorono\"i diagram of size $N$, which in turn can be derived from a minimal covering of the unit sphere by unit closed balls centered on the sphere.  As a consequence, we define for $N\geq N_{d, p, \left|\cdot\right|}$ a quantization based distance $\mathcal{Q}_{N,p}\coloneqq\left\Vert e_{N,p}(\mu, \cdot)- e_{N,p}(\nu, \cdot)\right\Vert_{\sup}$ which we will prove to be topologically equivalent to the Wasserstein distance $\mathcal{W}_{p}$.

In Section~\ref{euclideancase}, we consider the quadratic case (\emph{i.e.} the order $p$=2) and extend the characterization result to probability distributions on a separable 
Hilbert space $H$ with the norm $\left|\cdot\right|_{H}$ induced by the inner product $(\cdot\mid\cdot)_{H}$. In this section, we will prove by a purely analytical method that $N_{H, 2, \left|\cdot\right|_{H}}=2$~\footnote{Since the dimension of the Hilbert space that we discuss in this section can be finite or infinite, we write directly $H$ instead of $d$ in the subscript of $N_{d, p, \left|\cdot\right|}$.} and the topological equivalence of Wasserstein distance $\mathcal{W}_{2}$ and the distance $\mathcal{Q}^H_{2, 2}(\mu, \nu)\coloneqq\left\Vert e_{2,2}(\mu, \cdot)- e_{2,2}(\nu, \cdot)\right\Vert_{\sup}$ on $\mathcal{P}_{2}(H)$.

Section~\ref{enr1} is devoted to the one-dimensional setting.  Quantization based characterization not yet covered by the discussion in Section~\ref{genernalcase} and Section~\ref{euclideancase} are established. Furthermore, we prove that $\mathcal{Q}_{1, 1}$  is a complete distance on $\mathcal{P}_{1}(\mathbb{R})$ and give a counterexample to show that the distances $\mathcal{Q}_{N, 2}$, $N\ge 2$ are   not complete on $\mathcal{P}_{2}(\mathbb{R})$ in Section~\ref{42}.

\subsection{Preliminaries on Wasserstein distance}\label{wasserstein}

Let $(S, d)$ be a general Polish metric space. The relation between weak convergence and convergence for the Wasserstein distance $\mathcal{W}_{p}$ (see Definition~\ref{defwas}) is recalled in Theorem~\ref{was1}. We recall below some useful facts about the $L^{p}$-Wasserstein distance that will be called upon further on. The first one is that, for every $p\in[1, +\infty)$, $\mathcal{W}_{p}$ is a distance on $\mathcal{P}_{p}(S)$ \big($\mathcal{W}_{p}^{p}$ if $p\in(0,1)$\big), see e.g.~\cite{villani2003topics}[Theorem 7.3] for the proof and~\cite{berti2015gluing} for a recent reference. Next, the metric space $\big(\mathcal{P}_{p}(S), \mathcal{W}_{p}\big)$ is separable and complete, see e.g.~\cite{bolley2008separability} for the proof.  More generally, we refer to ~\cite{villani2008optimal}[Chapter 6] for an in depth presentation of Wasserstein distance and its properties. 
\begin{thm}(see~\cite{villani2003topics}[Theorem 7.12])\label{was1}
Let $\mu_{n} \!\in \mathcal{P}_{p}(S)$ for every $n\!\in\mathbb{N}^{*}\cup\{\infty\}$. Let $p\in[1,+\infty)$. Then,\\
\noindent$(a)$ $\mathcal{W}_{p}(\mu_{n},\mu_{\infty})\rightarrow0\;\text{if and only if}\;\begin{cases}
(\alpha)\;\mu_{n}\xRightarrow{(S)}\mu_{\infty}\\
(\beta)\;\exists \,x_{0}\!\in S, {\int_{S}d(x_{0},\xi)^{p}\mu_{n}(d\xi)\rightarrow\int_{S}d(x_{0},\xi)^{p}\mu_{\infty}(d\xi)}\\
\end{cases}\hspace{-0.4cm}.$\\
\noindent$(b)$ If
\begin{equation}\label{unifinte}
\exists\,x_{0}\!\in S,\;\;\;\lim_{R\rightarrow+\infty}\sup_{n\geq1}\int_{d(x_{0}, \xi)^{p}\geq R}d(x_{0}, \xi)^{p}\mu_{n}(d\xi)=0,
\end{equation}
then $(\mu_{n})_{n\geq1}$ is relatively compact for the Wasserstein distance $\mathcal{W}_{p}$. 
\end{thm}

\section{General quantization based characterizations on $\mathbb{R}^{d}$}\label{genernalcase}

This section is devoted to establish a general criterion that positively answers to \textit{Questions 1 and 2} in any dimension $d$, for any order $p$ and any norm on $\mathbb{R}^{d}$. 
The idea is to design an approximate identity $(\varphi_{\varepsilon})_{\varepsilon>0}$\footnote{By approximate identity we mean $\varphi_{\varepsilon}\!\in L^{1}\big(\mathbb{R}^{d}, \mathcal{B}(\mathbb{R}^{d}), \lambda_{d}\big)$, $\varepsilon>0$, such that $\int_{\mathbb{R}^{d}}\varphi_{\varepsilon}d\lambda_{d}=1$, $\sup_{\varepsilon>0}\int_{\mathbb{R}^{d}}\left|\varphi_{\varepsilon}\right|d\lambda_{d}<+\infty$ and $\lim_{\varepsilon\rightarrow0}\int_{\{\left|\xi\right|>\eta\}}\varphi_{\varepsilon}(\xi)\lambda_{d}(\xi)=0$ for every $\eta>0$.} based on the quantization error function $e_{N,p}(\mu,\cdot)$. 
Our construction of $(\varphi_{\varepsilon})_{\varepsilon>0}$ relies on a purely geometrical idea: it is based on a specified Vorono\"i diagram containing a bounded open Vorono\"i cell that we introduce in Section~\ref{voronoi}. The static characterization is established in Theorem~\ref{charstat}. Furthermore, Theorem~\ref{charcvg} shows that a pointwise convergence of the quantization error functions is enough to imply the $\mathcal{W}_{p}$-convergence of a $\mathcal{P}_{p}(\mathbb{R}^{d})$-valued sequence. 

\subsection{A review of Vorono\"i diagram, existence of bounded cells}\label{voronoi}
Let $\Gamma=\{x_{1}, \dots,x_{N}\}$ be a grid of size $N$. The \textit{Vorono\"i cell} generated by $x_{i}\!\in \Gamma$ is defined by 
\begin{equation}\label{voronoidef}
V_{x_{i}}(\Gamma)=\big{\{}\xi\!\in\mathbb{R}^{d}:\left|\xi-x_{i}\right|=\min_{1\leq j\leq N}\left|\xi-x_{j}\right|\big{\}},
\end{equation}
and $\big(V_{x_{i}}(\Gamma)\big)_{1\leq i\leq N}$ is called the \textit{Vorono\"i diagram} of $\Gamma$, which is a finite covering of $\mathbb{R}^{d}$ (see~\cite{graf2000foundations}). A Borel measure partition $\big(C_{x_{i}}(\Gamma)\big)_{1\leq i\leq N}$ is called a \textit{Vorono\"i partition} of $\mathbb{R}^{d}$ induced by $\Gamma$ if for every $i\!\in\{1,\dots,N\},\;C_{x_{i}}(\Gamma)\subset V_{x_{i}}(\Gamma)$.
We also define the \textit{open Vorono\"i cell} generated by $x_{i}\!\in \Gamma$ by 
\begin{equation}\label{voroopendef}
V^{o}_{x_{i}}(\Gamma)=\big{\{}\xi\!\in\mathbb{R}^{d}:\left|\xi-x_{i}\right|<\min_{1\leq j\leq N, j\neq i}\left|\xi-x_{j}\right|\big{\}}.
\end{equation}

If the norm $|\cdot|$ on $\mathbb{R}^{d}$ is strictly convex, we have $\mathring{V}_{x_{i}}(\Gamma)= V^{o}_{x_{i}}(\Gamma)\;\;\text{and}\;\;\overline{V^{o}_{x_{i}}(\Gamma)}=V_{x_{i}}(\Gamma)$, 
where $\mathring{A}$ and $\overline{A}$ denote the interior and the closure of $A$. Examples of strictly convex norms are the isotropic $\ell_{r}$-norms for $1<r<+\infty$ defined by $\left|(a^1,\ldots, a^d)\right|_{r}=\big(\left|a^{1}\right|^{r}+\dots+\left|a^{d}\right|^{r}\big)^{1/r}$. However, this is not true for any norm on $\mathbb{R}^{d}$, typically not for the $\ell^{1}$-norm (see~\cite{graf2000foundations}[Figure 1.2]) or the $\ell^{\infty}-$norm. 

We recall that $A\subset\mathbb{R}^{d}$ is star-shaped with respect to $a\!\in A$ if for every $b\!\in A$ and any $\lambda\in[0,1]$, $a+\lambda(b-a)\!\in A$. 

\begin{prop}\label{starshaped}(see~\cite{graf2000foundations}[Proposition 1.2])
Let $\Gamma=\{x_{1}, \dots,x_{N}\}$ be a grid of size $N\geq1$. For every $i\!\in\{1, \dots, N\}$, $V_{x_{i}}(\Gamma)$ and $V^{o}_{x_{i}}(\Gamma)$ are star-shaped relative to $x_{i}$.
\end{prop}

Now we discuss a sufficient condition to obtain a Vorono\"i diagram containing a bounded open Vorono\"i cell. The first result  in this direction is a rewriting Proposition 1.10 in~\cite{graf2000foundations}  for  Euclidean norms (stated here in view of our applications). 

\begin{prop}[$|\cdot|$ Euclidean norm]\label{prop:G&Lbounded} Let $(b_1, \ldots,b_{d+1})$ be an affine basis of $\mathbb{R}^{d}$ and let $b_0\!\in \mathring{\wideparen{\mathrm{Conv}(\{b_1, \ldots,b_{d+1}\})}}\neq\varnothing $. Set $\Gamma= \{0, b_1-b_0, \ldots, b_{d+1}-b_0\}$. 
Then, the  open Vorono\"i cell $V^{o}_{0}(\Gamma)$ generated by $0$ is bounded.
\end{prop}
\vspace{-0.2cm}
Let us provide now a geometrical criterion for a general norm  $|\cdot|$ on $\mathbb{R}^{d}$, let $\bar{B}_{|\cdot|}(x, r)$ denote the closed ball centered at $x$ with radius $r$ and let $S_{|\cdot|}(x, r)$ denote its sphere. 

\begin{prop}\label{recover}
Let $a_{1}, \dots, a_{k}\!\in S_{|\cdot|}(0,1)$ such that $S_{|\cdot|}(0,1)\subset\bigcup_{i=1}^{k}\bar{B}_{|\cdot|}(a_{i},1)$ (such a covering exists since $S_{|\cdot|}(0,1)$ is compact). If we choose $\Gamma=\{0, a_{1}, \dots, a_{k}\}$, then the Vorono\"i open set $V^{o}_{0}(\Gamma)\subset\bar{B}_{|\cdot|}(0, 1)$ and $\lambda_{d}\big(V^{o}_{0}(\Gamma)\big)>0$.
\end{prop}

\begin{proof}
As $S_{|\cdot|}(0,1)\subset\bigcup_{i=1}^{k}\bar{B}_{|\cdot|}(a_{i},1)$, for every $\xi\!\in S_{|\cdot|}(0, 1)$, there exists $j\!\in\{1, \dots, k\}$ such that $\left|\xi-a_{j}\right|\leq1=\left|\xi\right|$. If $\Gamma=\{0, a_{1}, \dots, a_k\}$,  then
\begin{equation}\label{unisphere}
\forall \xi\!\in S_{|\cdot|}(0, 1), \;\;\;\exists\,j\!\in\{1, \dots, k\}\;\text{such that}\;\xi\!\in V_{a_{j}}(\Gamma).
\end{equation}

Assume that  there exists $\xi\!\in V^{o}_{0}(\Gamma)\setminus\bar{B}_{|\cdot|}(0, 1)$. Since $V^{o}_{0}(\Gamma)$ is star-shaped relatively to $0$ and $\frac{1}{\left|\xi\right|}\in(0,1)$, we have
$\frac{\xi}{\left|\xi\right|}\!\in S_{|\cdot|}(0,1)\cap V^{o}_{0}(\Gamma)$. 
This contradicts (\ref{unisphere}) since $V_{0}^{o}(\Gamma)\cap V_{a_{j}}(\Gamma)\neq\varnothing,\; j=1,\dots, k. $
Consequently, $V^{o}_{0}(\Gamma)\subset\bar{B}_{|\cdot|}(0, 1)$. Finally,   $V^{o}_{0}(\Gamma)$ is an open set containing $0$, therefore, $\lambda_{d}\big(V^{o}_{0}(\Gamma)\big)>0$.
\end{proof}
The idea of the above proposition is to cover the unit sphere centered at the origin by a finite number of unit balls centered on the unit sphere. This leads us to introduce the following definition. 
\begin{defn}
We define the {\em minimal sphere covering number} $c(d,|\cdot|)$ as follows, 
\[c(d,|\,\cdot\,|)\coloneqq\min\Big\{k:\;\exists\{a_{1},\dots, a_{k}\}\subset S_{|\cdot|}(0,1)\;\text{such that}\;S_{|\cdot|}(0,1)\subset\bigcup_{i=1}^{k}\bar{B}_{|\cdot|}(a_{i},1)\Big\}<+\infty.\]
\end{defn}
\vspace{-0.2cm}
\noindent The index $c(d,|\cdot|)$ is finite since the unit sphere is a compact set in $\mathbb{R}^{d}$. Among all the possible norms, we will focus on the isotropic $\ell_{r}$-norms on $\mathbb{R}^{d}$. We show some examples of the minimal covering number $c(d, |\cdot|_{r})$ in the following proposition (whose proof is postponed to Appendix).
\begin{prop}
\label{ldpprop}
\begin{enumerate}[$(i)$]
\item $c(1, |\cdot|)=2$, where $|\cdot|$ denotes the absolute value.
\item $c(2, |\cdot|_{1})=2$ and $c(2, |\cdot|_{r})=3$ for every $1<r<+\infty$.
\item $c(d, |\cdot|_{\infty})=2$ for every dimension $d$.
\item Let $r\geq1$   such that $2^{r}\ge d$, then $c(d, |\cdot|_{r})\leq2d$.
\end{enumerate}
\end{prop}

\subsection{A general condition for probability measure characterization}\label{22d}

Let $\Gamma=\{x_{1}, \dots,x_{N}\}$ be a grid in which there exists at least an $x_{i_{0}}\!\in\Gamma$ such that the open Vorono\"i cell $V^{o}_{x_{i_{0}}}(\Gamma)$ is bounded and non-empty. Based on such a grid, one can construct an approximate identity as follows. Let  $\varphi : \mathbb{R}^{d}\rightarrow\mathbb{R}_{+}$ be the function defined by $ \displaystyle\varphi(\xi)=\min_{a\in\Gamma\setminus\{x_{i_{0}}\}}\left|\xi-a\right|^{p}-\min_{a\in\Gamma}\left|\xi-a\right|^{p}$.
The function $\varphi$ is clearly nonnegative, continuous and $\{\varphi>0\}=V^{o}_{x_{i_{0}}}(\Gamma)$ so that supp$(\varphi)=\overline{V^{o}_{x_{i_{0}}}(\Gamma)}$ is compact. Hence, $\int\varphi\, d\lambda_{d}\in(0, +\infty)$ since $\varphi(x_{i_{0}})=d\big(x_{i_{0}},\Gamma\setminus\{x_{i_{0}}\}\big)>0$ and we can normalize $\varphi$ by setting $\varphi_{1}(\xi)\coloneqq\frac{\varphi(x_{i_{0}}+\xi)}{\int\varphi d\lambda_{d}}$. For every $\varepsilon>0$, we define
$\varphi_{\varepsilon}(\xi)\coloneqq\frac{1}{\varepsilon^{d}}\varphi_{1}\Big(\frac{\xi}{\varepsilon}\Big)$, then $(\varphi_{\varepsilon})_{\varepsilon>0}$ is clearly an approximate identity (see~\cite{grafakos2008classical}[Section 1.2.4]). 

The following theorem gives conditions on the $L^{p}$-quantization error function to characterize a probability measure.

\begin{thm}[Static characterization]\label{charstat}
Let $p\in[1,+\infty)$,  let $|\cdot|$ be a norm on $\mathbb{R}^{d}$ and let $N\geq c(d, |\cdot|)+1$,  or $N\ge d+2$ if $|\cdot|$ is Euclidean. Then, the answer to Question~1 is positive i.e. 
if there exists a constant $C$ such that $e_{N,p}^{p}(\mu,\cdot)=e_{N,p}^{p}(\nu,\cdot)+C$, $\mu,\nu\!\in \mathcal{P}_{p}(\mathbb{R}^{d})$, then $\mu=\nu$. The constant $C$ is a posteriori $0$.
\end{thm}
\begin{proof}    
Following Proposition~\ref{prop:G&Lbounded} and~\ref{recover}, we choose a grid $\Gamma=\{0, a_{1},\dots ,a_{N-1}\}$ such that
$V^{o}_{0}(\Gamma)$ is bounded and $\lambda_{d}\big(V^{o}_{0}(\Gamma)\big)>0$. We define $\varphi: \mathbb{R}^{d}\rightarrow \mathbb{R}_{+}$, by $\varphi(\xi)=\min_{a\in\Gamma\setminus\{0\}}\left|\xi-a\right|^{p}-\min_{a\in\Gamma}\left|\xi-a\right|^{p}=\big(\min_{a\in\Gamma\setminus\{0\}}\left|\xi-a\right|^{p}-\left|\xi\right|^{p}\big)_{+}$  
and $(\varphi_{\varepsilon})_{\varepsilon>0}$ by $\varphi_{\varepsilon}(\xi)\coloneqq\frac{1}{C_{\varphi}\varepsilon^{d}}\varphi\Big(\frac{\xi}{\varepsilon}\Big)$, 
where $C_{\varphi}=\int\varphi \,d\lambda_{d}$. For any $x\!\in\mathbb{R}^{d}$,
\begin{align}
\varphi_{\varepsilon}*\mu(x)&=\int_{\mathbb{R}^{d}}\varphi_{\varepsilon}(x-\xi)\mu(d\xi)=\int_{\mathbb{R}^{d}}\frac{1}{\varepsilon^{d}}\frac{\varphi(\frac{x-\xi}{\varepsilon})}{\int\varphi d\lambda_{d}}\mu(d\xi)\nonumber\\
&=\frac{1}{C_{\varphi}\varepsilon^{d}}\int_{\mathbb{R}^{d}}\bigg(\min_{a\in\Gamma\setminus\{0\}}\left|\frac{x-\xi}{\varepsilon}-a\right|^{p}-\min_{a\in\Gamma}\left|\frac{x-\xi}{\varepsilon}-a\right|^{p}\bigg)\mu(d\xi)\nonumber\\
&=\frac{1}{C_{\varphi}\varepsilon^{d+p}}\bigg[\int_{\mathbb{R}^{d}}\min_{a\in\Gamma\setminus\{0\}}\left|x-\varepsilon a-\xi\right|^{p}\mu(d\xi)-\int_{\mathbb{R}^{d}}\min_{a\in\Gamma}\left|x-\varepsilon a-\xi\right|^{p}\mu(d\xi)\bigg].\nonumber
\end{align}

If we define two $N$-tuples $\tilde{x}$ and $\tilde{x}_{0}$ as $\tilde{x}=(x-\varepsilon a_{1},x-\varepsilon a_{1},x-\varepsilon a_{2},\dots,x-\varepsilon a_{N-1})$ and $\tilde{x}_{0}=(x,x-\varepsilon a_{1},x-\varepsilon a_{2},\dots,x-\varepsilon a_{N-1})$, then 
\[\int_{\mathbb{R}^{d}}\min_{a\in\Gamma\setminus\{0\}}\left|x-\varepsilon a-\xi\right|^{p}\mu(d\xi)=e_{N,p}^{p}(\mu,  \tilde{x})\;\text{and}\;\int_{\mathbb{R}^{d}}\min_{a\in\Gamma}\left|x-\varepsilon a-\xi\right|^{p}\mu(d\xi)=e_{N,p}^{p}(\mu,  \tilde{x}_{0}).\] 
Hence,  $\varphi_{\varepsilon}*\mu(x)=\frac{1}{C_{\varphi}\varepsilon^{d+p}}\big(e_{N,p}^{p}(\mu,  \tilde{x}) - e_{N,p}^{p}(\mu,  \tilde{x}_{0})\big)$.

The assumption $e_{N,p}^{p}(\mu,\cdot)=e_{N,p}^{p}(\nu,\cdot)+C$ implies that $e_{N,p}^{p}(\mu, \tilde{x})-e_{N,p}^{p}(\mu, \tilde{x}_{0}) =e_{N,p}^{p}(\nu, \tilde{x})-e_{N,p}^{p}(\nu, \tilde{x}_{0})$,
so that, for every $x\!\in\mathbb{R}^{d}$ and every $\varepsilon>0$, $\varphi_{\varepsilon}*\mu(x)=\varphi_{\varepsilon}*\nu(x)$. 
        
        One can finally conclude that $\mu=\nu$ by letting $\varepsilon\rightarrow0$ since $(\varphi_{\varepsilon})_{\varepsilon>0}$ is an approximate identity (see~\cite{rudin1991functional}[Theorem 6.32]).     Hence $C=0$.
\end{proof}

The following theorem shows that the pointwise convergence of the $L^{p}$-mean quantization error function is a necessary and sufficient condition for $\mathcal{W}_{p}$-convergence of probability distributions in $\mathcal{P}_{p}(\mathbb{R}^{d})$.

\begin{thm}[$\mathcal{W}_{p}$-convergence characterization]\label{charcvg}
Let $p\in[1,+\infty)$ and let $|\cdot|$ be any norm on $\mathbb{R}^{d}$. Let $\mu_{n} \!\in \mathcal{P}_{p}(\mathbb{R}^{d})$ for $n\!\in\mathbb{N}^{*}\!\cup\!\{\infty\}$.   The following properties are equivalent:
\begin{enumerate}[$(i)$]
\item $\mathcal{W}_{p}(\mu_{n}, \mu_{\infty})\xrightarrow{\;n\rightarrow+\infty\;}0$,
\item $\forall N\geq1,\;\;\;e_{N,p}(\mu_{n}, \cdot)\xrightarrow{\;n\rightarrow+\infty\;}e_{N,p}(\mu_{\infty}, \cdot)$ uniformly on $\mathbb{R}^{d}$,
\item $\exists\,N\geq c(d,|\cdot|)+1 \mbox{ or } N\ge d+2 \mbox{ if } |\cdot| \mbox{ is Euclidean  such that},\; e_{N,p}(\mu_{n}, \cdot)\xrightarrow{\;n\to+\infty\;}e_{N,p}(\mu_{\infty}, \cdot)$ pointwise on $\mathbb{R}^{d}$.
\end{enumerate}
\end{thm}


\begin{proof}[Proof of Theorem~\ref{charcvg}]
$(i)\Rightarrow(ii)$ is obvious from $(\ref{06})$.

\noindent$(ii)\Rightarrow(iii)$ is obvious.

\noindent$(iii)\Rightarrow(i)$ First of all, it follows from the convergence $e_{N,p}(\mu_{n},\cdot)\xrightarrow{n\rightarrow+\infty}e_{N,p}(\mu,\cdot)$ that
\begin{equation}\label{cvgmoment}
\hskip-0.1 cm e_{N,p}^{p}(\mu_{n},\textbf{0})\xrightarrow{\,n\rightarrow+\infty\,}e_{N,p}^{p}(\mu_{\infty},\textbf{0})\text{ i.e. } \int_{\mathbb{R}^{d}}\!\!\!\left|\xi\right|^{p}\mu_{n}(d\xi)\xrightarrow{\,n\rightarrow+\infty\,}\int_{\mathbb{R}^{d}}\!\!\!\left|\xi\right|^{p}\mu_{\infty}(d\xi)<+\infty,
\end{equation}
where $\textbf{0}=(0, \dots, 0)$. In particular, the sequence $\Big(\int_{\mathbb{R}^{d}}\left|\xi\right|^{p}\mu_{n}(d\xi)\Big)_{n\geq1}$ is bounded. Hence, the sequence of probability measures $(\mu_{n})_{n\geq1}$ is tight. 

Let $\widetilde{\mu}_{\infty}$ be a weak limiting probability distribution of $(\mu_{n})_{n\geq1}$ i.e. there exists a subsequence $\alpha(n)$ of $n$ such that $\mu_{\alpha(n)}\xRightarrow{(\mathbb{R}^{d})}\widetilde{\mu}_{\infty}$ as $n\rightarrow+\infty$.

Let $x=(x_{1}, \dots, x_{N})$ be any $N$-tuple in $(\mathbb{R}^{d})^{N}$. We define a continuous function $f_x:\mathbb{R}^{d}\rightarrow\mathbb{R}$ by $f_x(\xi)\coloneqq\min_{1\le i \le N}\left|\xi-x_{i}\right|^{p}-\left|\xi\right|^{p}$.
 Hence, owing to the elementary inequality $v^{p}-u^{p}\leq pv^{p-1}(v-u)$ for any $0\leq u\leq v<+\infty$, we derive 
 \begin{equation}\label{equiintegrable}
 \big|f_x(\xi)\big|\leq\max_{i \in\{1,\dots,N\}}p\big(\left|\xi\right|+\left|x_{i}\right|\big)^{p-1}\left|x_{i}\right|\leq C_{x, p}(1+\left|\xi\right|^{p-1}),
 \end{equation}
 where $C_{x, p}$ is a constant  depending on  $x$ and $p$.
 
Owing to  (\ref{cvgmoment}) and (\ref{equiintegrable}),  the sequence $\big(\int f_x^{\frac{p}{p-1}}d\mu_{n}\big)_{n\geq1}$ is bounded, hence $f_x$ is uniformly integrable with respect to $(\mu_{n})_{n\geq1}$ since $\frac{p}{p-1}>1$, so that $f_x$ is uniformly integrable with respect to any subsequence $(\mu_{\alpha(n)})_{n\geq1}$.
It follows that $\int_{\mathbb{R}^{d}}f_x(\xi)\mu_{\alpha(n)}(d\xi)\rightarrow\int_{\mathbb{R}^{d}}f_x(\xi)\widetilde{\mu}_{\infty}(d\xi),$ as $n\rightarrow+\infty$,
where
\begin{align}
&\int_{\mathbb{R}^{d}}f_x(\xi)\mu_{\alpha(n)}(d\xi)=\int_{\mathbb{R}^{d}}\big(\min_{i\in\{1,\dots,N\}}\left|\xi-x_{i}\right|^{p}-\left|\xi\right|^{p}\big)\mu_{\alpha(n)}(d\xi)=e_{N,p}^{p}(\mu_{\alpha(n)}, x) - e_{N,p}^{p}(\mu_{\alpha(n)}, \textbf{0}),\nonumber\\
&\text{and }\int_{\mathbb{R}^{d}}f_x(\xi)\widetilde{\mu}_{\infty}(d\xi)=e_{N,p}^{p}(\widetilde{\mu}_{\infty}, x) - e_{N,p}^{p}(\widetilde{\mu}_{\infty}, \textbf{0}).\nonumber
\end{align}

On the other hand,
$e_{N,p}^{p}(\mu_{\alpha(n)}, x) - e_{N,p}^{p}(\mu_{\alpha(n)}, \textbf{0})$ converges to $e_{N,p}^{p}(\mu_{\infty}, x) - e_{N,p}^{p}(\mu_{\infty}, \textbf{0})$ owing to the pointwise convergence in $(iii)$ at $\textbf{0}=(0,\dots,0)$ and $x=(x_{1}, \dots, x_{N})$.

Therefore, $e_{N,p}^{p}(\widetilde{\mu}_{\infty}, x) - e_{N,p}^{p}(\widetilde{\mu}_{\infty}, \textbf{0})= e_{N,p}^{p}(\mu_{\infty}, x) - e_{N,p}^{p}(\mu_{\infty}, \textbf{0})$,
which implies that for every $x\in(\mathbb{R}^{d})^{N},\;e_{N,p}^{p}(\widetilde{\mu}_{\infty}, x) - e_{N,p}^{p}(\mu_{\infty}, x)= C$,
where $C=e_{N,p}^{p}(\widetilde{\mu}_{\infty}, \textbf{0})-e_{N,p}^{p}(\mu_{\infty}, \textbf{0})$ is a real constant.
It follows from Theorem~\ref{charstat} that $\widetilde{\mu}_{\infty}=\mu_{\infty}$, which implies that $\mu_{\infty}$ is the the only limiting distribution of $(\mu_{n})_{n\geq1}$ for the weak convergence and consequently $\mu_{n}\xRightarrow{(\mathbb{R}^{d})}\mu$. We have already proved that $\int_{\mathbb{R}^{d}}\left|\xi\right|^{p}\mu_{n}(d\xi)\xrightarrow{n\rightarrow+\infty}\int_{\mathbb{R}^{d}}\left|\xi\right|^{p}\mu_{\infty}(d\xi)$ from ($\ref{cvgmoment}$), which finally shows that $\mathcal{W}_{p}(\mu_{n},\mu_{\infty})\xrightarrow{n\rightarrow+\infty}0$ owing to Theorem~\ref{was1}. \end{proof}

A careful reading of the proof shows that the following ``\`a la Paul L\'evy'' characterization result holds for limiting functions of $L^{p}$-quantization error functions.

\begin{cor}
Let $p\in[1+\infty)$. Let  $(\mu_{n})_{n\geq1}$ be a $\mathcal{P}_{p}(\mathbb{R}^{d})$-valued sequence. If 
\[
e_{N, p}(\mu_{n}, \cdot)\xrightarrow{n\rightarrow+\infty} f \;\text{ pointwise for some $N$ such that {\em static characterization} holds true}
\]
(Question~1), then there exists $\mu_{\infty}\!\in\mathcal{P}_{p}(\mathbb{R}^{d})$ such that $\mu_n \xRightarrow{(\mathbb{R}^{d})} \mu_\infty$ as $n\to +\infty$ and 
\begin{align}
f^p=e_{N, p}^{p}(\mu_{\infty}, \,\cdot\,)+\lim_{n}\int_{\mathbb{R}^{d}}\left|\xi\right|^{p}\mu_{n}(d\xi)-\int_{\mathbb{R}^{d}}\left|\xi\right|^{p}\mu_{\infty}(d\xi).\nonumber
\end{align}

\end{cor}

\vskip-.1cm 
\noindent
Now we will take advantage of what precedes to  introduce   a {\em quantization based distance} on $\mathcal{P}_{p}(\mathbb{R}^{d})$. 
Let $\mathcal{C}_{b}\big((\mathbb{R}^{d})^{N}, \mathbb{R}\big)$ denote the space of bounded $\mathbb{R}$-valued continuous functions defined on $(\mathbb{R}^{d})^{N}$ equipped with the sup norm  $\left\Vert\cdot\right\Vert_{\sup}$. Let $p\in[1,+\infty)$. If $\mu\!\in\mathcal{P}_{p}(\mathbb{R}^{d}),\,e_{N, p}(\mu, \cdot)-e_{N, p}(\delta_{0}, \cdot)\!\in\mathcal{C}_{b}\big((\mathbb{R}^{d})^{N}, \mathbb{R}\big)$  (note   $ e_{N, p}\big(\delta_{0}, (x_1,\ldots,x_N)\big) $ $\displaystyle = \min_{i=1,\ldots,N} |x_i|$) since  inequality (\ref{control}) implies that $\left\Vert e_{N, p}(\mu, \cdot)-e_{N, p}(\delta_{0}, \cdot)\right\Vert_{\sup}\!\leq\mathcal{W}_{p}(\mu, \delta_{0})\!=\!\Big[\int_{\mathbb{R}^{d}}\!\left|\xi\right|^{p}\mu(d\xi)\Big]^{1/p}\!<\!+\infty$. Then,  we define a function $\mathcal{Q}_{N, p}$ on $\mathcal{P}_{p}(\mathbb{R}^{d}) $ is defined  by
\begin{align}\label{Qnpdisdef}
(\mu, \nu)\longmapsto \mathcal{Q}_{N, p}(\mu, \nu)\coloneqq&\left\Vert \big(e_{N,p}(\mu, \cdot) - e_{N,p}(\delta_{0}, \cdot)\big) - \big(e_{N,p}(\nu, \cdot) - e_{N,p}(\delta_{0}, \cdot)\big)\right\Vert_{\sup}\nonumber\\
=&\left\Vert e_{N,p}(\mu, \cdot) - e_{N,p}(\nu, \cdot)\right\Vert_{\sup}.
\end{align}
For any $\mu, \nu\!\in\mathcal{P}_{p}(\mathbb{R}^{d})$, inequality~\eqref{control} implies $\mathcal{Q}_{N, p}(\mu, \nu)\leq\mathcal{W}_{p}(\mu, \nu)<+\infty$ so that $\mathcal{Q}_{N, p}(\mu, \nu)\in[0, +\infty)$. 
Combining Theorems~\ref{charstat} and~\ref{charcvg} implies  the following result.

\begin{cor}\label{cola}Let $p\in[1, +\infty)$.

\noindent $(a)$  $N_{d, p, |\cdot|}\leq c(d,|\cdot|)+1$ for any norm and $N_{d, p, |\cdot|}\le d+2$ if $|\cdot|$ is Euclidean.

\noindent$(b)$  If $N\geq c(d,|\cdot|)+1$ or $N\ge d+2$ if $|\cdot|$ is Euclidean, then $\mathcal{Q}_{N, p}$ defined by (\ref{Qnpdisdef}) is a distance on $\mathcal{P}_{p}(\mathbb{R}^{d})$ and $\mathcal{Q}_{N, p}$ is topologically equivalent to the Wasserstein distance $\mathcal{W}_{p}$.
\end{cor}

\vskip-0.10cm 
\noindent {\em Comments on optimality.} If we consider only the quadratic case $p=2$ and a norm $|\cdot|$ induced by an inner product, the result in Corollary~\ref{cola}-$(a)$ is in fact not optimal. In the next section, 
we will prove that in such a setting, $N_{d, 2, \left|\cdot\right|}=2$ and this result can also be extended to any separable (possibly infinite-dimensional) Hilbert space. 

\section{Quadratic quantization based characterization on a separable Hilbert space: $N_{H, 2}=2$ }\label{euclideancase}

Let $H$ denote a separable Hilbert space with inner product $(\cdot\mid\cdot)_{H}$. Let $|\cdot|_{H}$ denote the norm on $H$ induced by $(\cdot\mid\cdot)_{H}$. When there is no ambiguity, we drop the index $_H$ and write $(\cdot\mid\cdot)$ and $|\cdot|$. The separable Hilbert space is a very common setup for applications, for example in functional data analysis: one can set  $H=L^{2}\big([0, T], dt\big)$ and $X=(X_{t})_{t\in[0, T]}$ a bi-measurable process such that $\int_0^T \mathbb{E} X^2_t\, dt <+\infty$. For more information about functional data analysis with an $L^{2}$-setup, we refer to~\cite{hsing2015theoretical} among others.

We first prove in the quadratic case ($p=2$), that both  static  (see further Proposition~\ref{e22prop}) and  $\mathcal{W}_{2}$-convergence  (see further Theorem~\ref{e22thm}) characterizations can be obtained at level $N=2$ by an analytical method. Then we will show that $N_{H, 2}:= N_{H, 2, \left|\cdot\right|_{H}}=2$ and for any $\mu, \nu\!\in\mathcal{P}_{2}(H)$, $\mathcal{Q}_{2, 2}(\mu, \nu)\coloneqq\left\Vert e_{2, 2}(\mu, \cdot)-e_{2, 2}(\nu, \cdot)\right\Vert_{\sup}$ is a well-defined distance on $\mathcal{P}_{2}(H)$ which is topologically equivalent to $\mathcal{W}_{2}$. 

\smallskip
Proofs of quadratic quantization based characterizations rely on the following lemma.

\begin{lem}\label{31}
$(a)$ Let $\mu, \nu\!\in\mathcal{P}_{2}(H)$. If for every $u\!\in H, \left|u\right|=1$, $\mu\circ\big(\xi\mapsto(\xi\mid u)\big)^{-1}=\nu\circ\big(\xi\mapsto(\xi\mid u)\big)^{-1},$
then $\mu=\nu$. \\
\noindent$(b)$ Let $\mu_{n}\!\in\mathcal{P}_{2}(H)$ for every $n\!\in\mathbb{N}^{*}\cup\{\infty\}$. If $\int_{H}\left|\xi\right|^{2}\mu_{n}(d\xi)\xrightarrow{\;n\rightarrow+\infty\;}\int_{H}\left|\xi\right|^{2}\mu_{\infty}(d\xi)$ and for every $u\!\in H, \,\left|u\right|=1$, $\mu_{n}\circ\big(\xi\mapsto(\xi\mid u)\big)^{-1}\xRightarrow{(\mathbb{R})}\mu_{\infty}\circ\big(\xi\mapsto(\xi\mid u)\big)^{-1}$, then $\mathcal{W}_{2}(\mu_{n}, \mu_{\infty})\rightarrow0$.
\end{lem}
\begin{proof}
As $(H, |\cdot|)$ is separable, let $(h_k)_{k\geq1}$ be a countable orthonormal basis of $(H, |\cdot|)$. 

\smallskip
\noindent$(a)$ Let $X, Y$ be random variables with respective distributions $\mu$ and $\nu$ and let $\lambda\!\in H$. We define for every $m\geq1$, $X^{(m)}\coloneqq \sum_{k=1}^{m} (X\,|\, h_{k})h_{k}$, $Y^{(m)}\coloneqq \sum_{k=1}^{m} (Y\,|\,h_{k})h_{k}$ and $\lambda^{(m)}\coloneqq \sum_{k=1}^{m}(\lambda\,|\,h_{k})h_{k}$.  
For  $m\geq1$, let $u^{(m)}=\frac{\lambda^{(m)}}{\left|\lambda^{(m)}\right|}$ (convention $\frac{0}{\left|0\right|}=0$), then we have
\begin{align}
(\lambda\,|\,X^{(m)})&=\sum_{k=1}^{+\infty}(\lambda\, |\, h_{k})(X^{(m)}\, | \,h_{k})=\sum_{k=1}^{m}(\lambda\, |\, h_{k})(X\, | \,h_{k})=\big|\lambda^{(m)}\big|\big(X\,\big{|}\,u^{(m)}\big).\nonumber
\end{align}

Similarly, $\small(\lambda\,|\,Y^{(m)})=\big|\lambda^{(m)}\big|\big(Y\,\big{|}\,u^{(m)}\big)$. Let $i$ be such that $i^{2}=-1$. It follows that \begin{align}
\mathbb{E}\,&e^{i(\lambda\mid X^{(m)})}=\mathbb{E}\,e^{i\left|\lambda^{(m)}\right|(X\mid u^{(m)})}
=\int_{H}e^{i\,|\lambda^{(m)}|\,\xi}\mu\circ\big(\xi\mapsto(u^{(m)}\mid \xi)\big)^{-1}(d\xi)\nonumber\\
&=\int_{H}e^{i\,|\lambda^{(m)}|\,\xi}\nu\circ\big(\xi\mapsto(u^{(m)}\mid \xi)\big)^{-1}(d\xi)=\mathbb{E}\,e^{i(\lambda\mid Y^{(m)})}.\nonumber
\end{align}
Since we can arbitrarily choose $\lambda$, we have for every $m\geq1$, $\mathrm{Law}(X^{(m)})=\mathrm{Law}(Y^{(m)})$.  Let $F: H\rightarrow \mathbb{R}$ be a bounded continuous function. Then, for every $m\geq1$, $\mathbb{E}\,F(X^{(m)})=\mathbb{E}\,F(Y^{(m)})$ which implies $\mathbb{E}\,F(X)=\mathbb{E}\,F(Y)$ by letting $m\rightarrow+\infty$. Hence, $\mu=\nu$. 

\smallskip
\noindent$(b)$ For every $n\geq1$, let $X_{n}$ be random variables with distribution $\mu_{n}$ and let $X_{\infty}$ be a random variable with distribution $\mu_{\infty}$. We define for every $n\geq1$ and for every $m\geq1$, $X_{n}^{(m)}\coloneqq \sum_{i=1}^{m} (X_{n} | h_{i})h_{i}$ and $X_{\infty}^{(m)}\coloneqq \sum_{i=1}^{m} (X_{\infty}| h_{i})h_{i}$. 
Following the lines of item~$(a)$, we get for every   $m\geq1$, $X_{n}^{(m)}\xRightarrow{(H)} X_{\infty}^{(m)}$ as $n\rightarrow+\infty$, since the convergence of characteristic function implies  weak convergence. 

Now, let $F: H\rightarrow \mathbb{R}$ be a Lipschitz continuous function with Lipschitz coefficient $[F]_{\text{Lip}}\coloneqq\sup_{x,y \in H}\frac{\left|F(x)-F(y)\right|}{\left|x-y\right|}$. For every (temporarily) fixed $m\geq1$,
\begin{align*}
&\lim_{n}\big|\mathbb{E}\,F(X_{n})-\mathbb{E}\,F(X_{\infty})\big|\\ &\hspace{0.3cm}\leq\lim_{n}\big|\mathbb{E}\,F(X_{n})-\mathbb{E}\,F(X_{n}^{(m)})\big|+\lim_{n}\big|\mathbb{E}\,F(X_{n}^{(m)})-\mathbb{E}\,F(X_{\infty}^{(m)})\big|+\big|\mathbb{E}\,F(X_{\infty}^{(m)})-\mathbb{E}\,F(X_{\infty})\big|\\
&\hspace{0.3cm}\leq\lim_{n}\big|\mathbb{E}\,F(X_{n})-\mathbb{E}\,F(X_{n}^{(m)})\big|+0+\big|\mathbb{E}\,F(X_{\infty}^{(m)})-\mathbb{E}\,F(X_{\infty})\big| \;\;\;\text{(since $X_{n}^{(m)}\xRightarrow{(H)} X_{\infty}^{(m)}$)}.
\end{align*}

Then, for every   $n\ge 1$, 
\begin{align}
\big|\mathbb{E}\,F(X_{n})-\mathbb{E}\,F(X_{n}^{(m)})\big|&\leq\mathbb{E}\big| F(X_{n})-F(X_{n}^{(m)})\big|\leq[F]_{\text{Lip}}\mathbb{E}\big|X_{n}-X_{n}^{(m)}\big|\leq[F]_{\text{Lip}}\big\Vert X_{n}-X_{n}^{(m)}\big\Vert_{2}.\nonumber
\end{align}
Similarly, we also have \small$\big|\mathbb{E}\, F(X_{\infty}^{(m)})-\mathbb{E}\,F(X_{\infty})\big|\leq [F]_{\text{Lip}}\big\Vert X_{\infty}-X_{\infty}^{(m)}\big\Vert_{2}$.

\normalsize It follows from Fatou's Lemma for the weak convergence and the convergence assumption made on $\mathbb{E} |X_{n} |^{2}$ that 
\begin{align}
\limsup_{n}&\,\big\Vert X_{n}-X_{n}^{(m)}\big\Vert_{2}^{2}=\limsup_{n}\mathbb{E}\big|X_{n}-X_{n}^{(m)}\big|^{2}=\limsup_{n}\Big[\mathbb{E}\big|X_{n}\big|^{2}-\mathbb{E}\big|X_{n}^{(m)}\big|^{2}\Big]\nonumber\\
&=\mathbb{E}\big|X_{\infty}\big|^{2}-\liminf_{n}\mathbb{E}\big|X_{n}^{(m)}\big|^{2}\leq\mathbb{E}\big|X_{\infty}\big|^{2}-\mathbb{E}\big|X_{\infty}^{(m)}\big|^{2}=\big\Vert X_{\infty}-X_{\infty}^{(m)}\big\Vert^{2}_{2}.
\end{align}
Hence, for every $m\geq1$,
\begin{align}
\lim_{n}&\big|\mathbb{E}\,F(X_{n})-\mathbb{E}\,F(X_{\infty})\big|\leq\limsup_{n}\,[F]_{\text{Lip}}\big\Vert X_{n}-X_{n}^{(m)}\big\Vert_{2}+[F]_{\text{Lip}}\big\Vert X_{\infty}-X_{\infty}^{(m)}\big\Vert_{2}\nonumber\\
&\leq2[F]_{\text{Lip}}\big\Vert X_{\infty}-X_{\infty}^{(m)}\big\Vert_{2}.\nonumber
\end{align}
Then, $\big\Vert X_{\infty}-X_{\infty}^{(m)}\big\Vert_{2}\rightarrow0$ as $m\rightarrow+\infty$ by the Lebesgue dominated convergence theorem since $\big|X_{\infty}-X_{\infty}^{(m)}\big|\leq\big|X_{\infty}\big|\!\in L^{2}(\mathbb{P})$ so that $\mathbb{E}\,F(X_{n})\rightarrow \mathbb{E}\,F(X_{\infty})$ as $n\rightarrow+\infty$.
Thus, $X_{n}\xRightarrow{\,(H)\,}X_{\infty}$ and we can conclude that $\mathcal{W}_{p}(\mu_{n}, \mu_{\infty})\rightarrow0$ by applying Theorem~\ref{was1}.
\end{proof}

\begin{prop}[Static characterization]\label{e22prop}
Let $\mu,\,\nu\!\in \mathcal{P}_{2}(H)$. If $e_{2,2}(\mu,\cdot)=e_{2,2}(\nu,\cdot)+C$ for some real constant $C$, then $\mu=\nu$ and $C=0$.
\end{prop}

\begin{proof}
Let $a,b\!\in H$, then $e_{2,2}^{2}\big(\mu,(a,b)\big)=\int_{H}\left|\xi-a\right|^{2}\wedge\left|\xi-b\right|^{2}\mu(d\xi)$.

As $e_{2,2}^{2}\big(\mu,(a,b)\big)=e_{2,2}^{2}\big(\nu,(a,b)\big)+C$ for every $(a,b)\!\in H^{2}$, in particular, if $a=b$, $\int_{H}\left|\xi-a\right|^{2}\mu(d\xi)=\int_{H}\left|\xi-a\right|^{2}\nu(d\xi)+C$.
Hence, using that $(x-y)_{+}=x-x\wedge y$, we have
\begin{equation}\label{plus}
\forall a,b\!\in H,\;\;\int_{H}\big(\left|\xi-a\right|^{2}-\left|\xi-b\right|^{2}\big)_{+}\mu(d\xi)=\int_{H}\big(\left|\xi-a\right|^{2}-\left|\xi-b\right|^{2}\big)_{+}\nu(d\xi).
\end{equation}
Note that $\left|\xi-a\right|^{2}-\left|\xi-b\right|^{2}=2\Big(b-a \;\Big| \;\xi-\frac{a+b}{2}\Big)$. Hence, if we take $a=\lambda u$ and $b=\lambda' u$ with $\lambda,\lambda'\!\in\mathbb{R},\;\lambda'>\lambda$ for some common $u\!\in H$ with $\left|u\right|=1$, we obtain
\[\big(\left|\xi-a\right|^{2}-\left|\xi-b\right|^{2}\big)_{+}=2(\lambda'-\lambda)\left((\xi\mid u)-\frac{\lambda+\lambda'}{2}\right)_{+}.\]

As a consequence of (\ref{plus}), we derive that
\begin{align}
\forall \lambda,\lambda'\!\in\mathbb{R},\lambda'>\lambda,\;\;\;\;\int_{H}\left((\xi\mid u)-\frac{\lambda+\lambda'}{2}\right)_{+}\!\!\!\mu(d\xi)=\int_{H}\left((\xi\mid u)-\frac{\lambda+\lambda'}{2}\right)_{+}\!\!\!\nu(d\xi).\nonumber
\end{align}

In turn, this implies, by letting $\lambda'\rightarrow\lambda$,
\begin{align}\label{aallpp}
\forall u\!\in H, \left|u\right|=1,\;\forall\lambda\!\in\mathbb{R},\qquad \int_{H}\Big((\xi\mid u)-\lambda\Big)_{+}\mu(d\xi)=\int_{H}\Big((\xi\mid u)-\lambda\Big)_{+}\nu(d\xi).
\end{align}

The function $\lambda\mapsto\big((\xi\mid u)-\lambda\big)_{+}$ is right differentiable with $\mathbbm{1}_{(\xi\mid u)>\lambda}$ as a right derivative and $\mu$-integrable. Hence, by the Lebesgue differentiation theorem, we can right differentiate the equality ($\ref{aallpp}$) which yields for every $u\!\in H, \left|u\right|=1$ and for every $\lambda\!\in\mathbb{R}$, $\mu\big((\xi\mid u)> \lambda\big)=\nu\big((\xi\mid u)> \lambda\big)$.

Hence, for every $u\!\in H, \left|u\right|=1$, $\mu\circ\big(\xi\mapsto(\xi\mid u)\big)^{-1}\!=\!\nu\circ\big(\xi\mapsto(\xi\mid u)\big)^{-1}$ since they have the same survival function.
We conclude by  Lemma~\ref{31}~$(a)$ that $\mu=\nu$~and~$C=0$.\end{proof}

\vskip -0.2cm
The following theorem shows the equivalence of $\mathcal{W}_{2}$-convergence of $(\mu_{n})_{n\geq1}$ in $\mathcal{P}_{2}(H)$ and the pointwise convergence of quadratic quantization error function $\big(e_{2,2}(\mu_{n}, \cdot)\big)_{n\geq1}$.

\begin{thm}[$\mathcal{W}_{2}$-convergence characterization]\label{e22thm}
Let $\mu_{n}\!\in\mathcal{P}_{2}(H)$ for every $n\!\in\mathbb{N}^{*}\cup\{\infty\}$.  The following properties are equivalent:
\begin{enumerate}[$(i)$]
\item $\mathcal{W}_{2}(\mu_{n},\mu_{\infty})\xrightarrow{\;n\rightarrow+\infty\;}0$,
\item $e_{2,2}(\mu_{n},\cdot)\xrightarrow{\;n\rightarrow+\infty\;}e_{2,2}(\mu_{\infty},\cdot)$ uniformly,
\item $e_{2,2}(\mu_{n},\cdot)\xrightarrow{\;n\rightarrow+\infty\;}e_{2,2}(\mu_{\infty},\cdot)$ pointwise.
\end{enumerate}
\end{thm}

Before proving Theorem~\ref{e22thm}, we recall the convergence of left and right derivatives of a converging sequence of convex functions. Let $\partial_{-}f$ (respectively $\partial_{+}f$) denote the left derivative ({resp.} {right derivative}) of a convex function $f$.

\begin{lem}\label{convex}(See e.g.~\cite{lackovic1982behaviour}[Theorems 2.5])
Let $f_{n} : \mathbb{R}^{d}\rightarrow\mathbb{R}^{d},n\!\in\mathbb{N}^{*},$ be a sequence of convex functions converging pointwise to a function $f : \mathbb{R}^{d}\rightarrow\mathbb{R}^{d}$. Let $G\coloneqq\{x\!\in\mathbb{R}\,|\,\partial_{-}f(x)\neq\partial_{+}f(x)\}$. Then for every point $x\!\in\mathbb{R}\setminus G$,  
\[\lim_{n}\partial_{+}f_{n}(x)=\lim_{n}\partial_{-}f_{n}(x)=f'(x).\]
\end{lem}

\begin{proof}[Proof of Theorem~\ref{e22thm}]

\noindent$(i)\Rightarrow(ii)$$\;\;\;$ is obvious from $(\ref{06})$.

\noindent$(ii)\Rightarrow(iii)$$\;\;$ is obvious.

\noindent$(iii)\Rightarrow(i)$$\;\;\;$ 
For every $(a,b)\!\in H^{2}$,
\begin{equation}
e_{2,2}^{2}\big(\mu_{n},(a,b)\big)=\int_{H}\left|\xi-a\right|^{2}\wedge\left|\xi-b\right|^{2}\mu_{n}(d\xi)\xrightarrow{n\rightarrow+\infty}e_{2,2}^{2}\big(\mu_{\infty},(a,b)\big)\!=\!\int_{H}\left|\xi-a\right|^{2}\wedge\left|\xi-b\right|^{2}\!\!\mu_{\infty}(d\xi).\nonumber
\end{equation}
In particular, $\forall a\!\in H,\;\int_{H}\left|\xi-a\right|^{2}\mu_{n}(d\xi)\xrightarrow{n\rightarrow+\infty}\int_{H}\left|\xi-a\right|^{2}\mu_{\infty}(d\xi)$.
Hence, using that $(x-y)_{+}=x-x\wedge y$, we get 
\[\forall a,b\!\in H,\;\;\int_{H}\big(\left|\xi-a\right|^{2}-\left|\xi-b\right|^{2}\big)_{+}\mu_{n}(d\xi)\xrightarrow{n\rightarrow+\infty}\int_{H}\big(\left|\xi-a\right|^{2}-\left|\xi-b\right|^{2}\big)_{+}\mu_{\infty}(d\xi).\]

Following the lines of the proof of Proposition~\ref{e22prop}, we get
\begin{equation}\label{cvgsimple}
\forall \lambda\!\in\mathbb{R},\, \forall u\!\in H, \left|u\right|=1,\;\int_{H}\Big((\xi\mid u)-\lambda\Big)_{+}\mu_{n}(d\xi)\xrightarrow{n\rightarrow+\infty}\int_{H}\Big((\xi\mid u)-\lambda\Big)_{+}\mu_{\infty}(d\xi).
\end{equation}

For $\mu\!\in\mathcal{P}_{2}(H)$ and $u\!\in S_{|\cdot|}(0,1)$, we define the real-valued convex function $\phi_{\mu}$ by $\phi_{\mu}\;:\;\;\lambda\mapsto\int\big((\xi\mid u)-\lambda\big)_{+}\mu(d\xi)$. It follows from (\ref{cvgsimple}) that $(\phi_{\mu_{n}})_{n\geq0}$ converges pointwise to $\phi_{\mu_{\infty}}$. Moreover, 
$\phi_{\mu_{n}}$, $\phi_{\mu_{\infty}}$ are right-differentiable and their right derivatives are given by $\partial_{+}\phi_{\mu_{n}}(\lambda)=\mu_{n}\big((\xi\mid u)>\lambda\big)$ and $\partial_{+}\phi_{\mu_{\infty}}(\lambda)=\mu_{\infty}\big((\xi\mid u)>\lambda\big)$ respectively. Note that the functions $1-\partial_{+}\phi_{\mu_{n}}$ and $1-\partial_{+}\phi_{\mu_{\infty}}$ are the cumulative distribution functions of the probability distributions $\mu_{n}\circ\big(\xi\mapsto(\xi\mid u)\big)^{-1}$ and $\mu_{\infty}\circ\big(\xi\mapsto(\xi\mid u)\big)^{-1}$ and that the set of discontinuity points of $1-\partial_{+}\phi_{\mu_{\infty}}$ and $\partial_{+}\phi_{\mu_{\infty}}$, is $G=\{\lambda\,:\,\mu_{\infty}\big(\{\xi:(\xi\mid u)=\lambda\}\big)>0\}$. 

We know from Lemma~\ref{convex} that for every $\lambda\!\in\mathbb{R}\setminus G$, $\partial_{+}\phi_{\mu_{n}}(\lambda)\xrightarrow{n\rightarrow+\infty}\partial_{+}\phi_{\mu_{\infty}}(\lambda)$ and that $\partial_{-}\phi_{\mu_{\infty}}$ is continuous on $\mathbb{R}\setminus G$. Hence  
\begin{equation}\label{probarprodu}
\forall u\!\in H, \left|u\right|=1,\;\;\;\;\;\mu_{n}\circ\big(\xi\mapsto(\xi\mid u)\big)^{-1}\xRightarrow{\;(\mathbb{R})\;}\mu_{\infty}\circ\big(\xi\mapsto(\xi\mid u)\big)^{-1}.
\end{equation}
Moreover, $e_{2,2}\big(\mu_{n},(0,0)\big)$ converges to $e_{2,2}\big(\mu_{\infty},(0,0)\big)$, which also reads $\int_{H}\left|\xi\right|^{2}\mu_{n}(d\xi)\rightarrow\int_{H}\left|\xi\right|^{2}\mu_{\infty}(d\xi)$.
Consequently, it follows from Lemma~\ref{31}-$(b)$ that $\mathcal{W}_{2}(\mu_{n},\mu_{\infty})\rightarrow0$ as $n\rightarrow+\infty$.
\end{proof}

\begin{rem} Proposition~\ref{e22prop} and Theorem~\ref{e22thm} directly imply that $N_{H, 2, \left|\cdot\right|_{2}}\leq2$. In fact, for every $a\!\in H$, 
\begin{align}
e_{1, 2}(\mu, a)=\int_{H}\left|\xi-a\right|_{H}^{2}\mu(d\xi)=\int_{H}\left|\xi\right|_{H}^{2}\mu(d\xi)-2\Big(\int_{H}\xi\mu(d\xi)\,\big|\,a\Big)_{H}+\left|a\right|_{H}^{2}.\nonumber
\end{align}

Thus, if $\mu, \nu\!\in\mathcal{P}_{2}(H)$ are such that 
\begin{align}\label{condmoment12}
\int_{H}\left|\xi\right|_{H}^{2}\mu(d\xi)=\int_{H}\left|\xi\right|_{H}^{2}\nu(d\xi)\;\;\text{and}\;\;\int_{H}\xi\mu(d\xi)=\int_{H}\xi\nu(d\xi),
\end{align}

then we have $e_{1, 2}(\mu, \cdot)=e_{1, 2}(\nu, \cdot)$. But condition~\eqref{condmoment12} is clearly not sufficient  to have $\mu=\nu$. Consequently, $N_{H, 2, \left|\cdot\right|_{2}}=2$. 
\end{rem}

Like what we did in Section~\ref{22d}, we define a function $\mathcal{Q}^{H}_{2, 2}$ on $\big(\mathcal{P}_{2}(H)\big)^{2}$ by $(\mu, \nu)\mapsto\mathcal{Q}^{H}_{2, 2}(\mu, \nu)=\left\Vert e_{2,2}(\mu, \cdot)-e_{2,2}(\nu, \cdot)\right\Vert_{\sup}.$
Then inequality $(\ref{control})$ implies that $\mathcal{Q}^{H}_{2, 2}(\mu, \nu)\in[0, +\infty)$. Moreover, Proposition~\ref{e22prop} and Theorem~\ref{e22thm} lead the following corollary.
\begin{cor}\label{cor:equidist} The distances 
$\mathcal{Q}^{H}_{2, 2}$ and  $\mathcal{W}_{2}$  are  topologically equivalent on $\mathcal{P}_{2}(H)$.
\end{cor}

We conclude this section by an ``\`A la Paul L\'evy'' characterization of a limit of quantization errors functions.
\normalsize
\begin{thm}[\`A la Paul L\'evy characterization]\label{thm:PlCarac} Let $(H,|\cdot|_H)$ be a separable Hilbert space. Let $(\mu_n)_{n\ge 1}$ be a $\mathcal{P}_2(H)$-valued sequence and let $f: H\to \mathbb{R}_+$ be such that  
\[
e_{2,2}(\mu_n,\cdot) \xrightarrow{n\rightarrow+\infty} f \; \mbox{ pointwise}.
\]
Then there exists $\mu_{\infty}\!\in \mathcal{P}_2(H)$ such that $\mu_n \stackrel{(H_{w})}{\Longrightarrow}\mu_{\infty}$  (where $(H_{w})$ stands for the   weak topology on $H$) and
\[
f^2 = e_{2,2}(\mu_n ,\cdot)^2 + \lim_n  \int_H |\xi|^2 \, \mu_{n}(d\xi)  - \int_H |\xi|^2 \, \mu_{\infty}(d\xi). 
\]
\end{thm}
\begin{proof}
The sequence $e_{2,2}\big(\mu_n,(0,0)\big)^2 = \int_H |\xi|^2\mu_n(d\xi)$, $n\ge 1$, is bounded, hence  the sequence $(\mu_n)_{n\ge 1}$ is tight for the weak topology $(H_{w})$ on $H$, which is metrizable since $H$ is separable (and generate the same Borel $\sigma$-field  as the strong one). Consequently there exists a subsequence $\mu_{\varphi(n)}\stackrel{(H_{w})}{\Longrightarrow}\mu_{\infty}\!\in \mathcal{P}_2(H)$ since the mapping $\xi \mapsto |\xi|^2$ is weakly  lower semi-continuous and non-negative. Now note that, for a fixed $x=(x_1,x_2)\!\in H^2$, the mapping $\xi \mapsto \min \big(|\xi-x_1|^2, |\xi-x_2|^2\big)-|\xi|^2 =\min \big(|x_1|^2-2(x_1|\xi),|x_2|^2-2(x_2|\xi)\big)$ is weakly continuous and $(\mu_n)_{n\ge 1}$-uniformly integrable since it is sublinear.  Hence 
\begin{align*}
e_{2,2}^2(\mu_{\varphi(n)}, x) \longrightarrow &\int_H\min \big(|x_1|^2-2(x_1|\xi),|x_2|^2-2(x_2|\xi)\big)\mu_{\infty}(d\xi) +f^2\big((0,0)\big)\mbox{ as } n\to +\infty\\
& =    e_{2,2}^2(\mu_\infty,x) + f^2\big((0,0)\big)- \int_H |\xi|^2 \, \mu_{\infty}(d\xi).
\end{align*} 
For two such limiting distributions $\mu_{\infty}$ and $\mu'_{\infty}$ it follows from what precedes that $e_{2,2}^2(\mu_{\infty}, \cdot)= e_{2,2}^2(\mu'_{\infty}, \cdot)+C_{\infty}$ for some real constant $C_{\infty}$. Hence $\mu_{\infty}= \mu'_{\infty}$ by Proposition~\ref{e22prop}, which in turn implies that $\mu_n \stackrel{(H_{w})}{\Longrightarrow}\mu_{\infty}$.
\end{proof}

\section{Further quantization based characterizations
on $\mathbb{R}$}\label{enr1}
Let $|\cdot|$ denote the absolute value on $\mathbb{R}$. Results from  Section~\ref{genernalcase} (Theorem~\ref{charstat} and~\ref{charcvg}, Proposition~\ref{ldpprop}-$(i)$) imply that $N_{1,p}:=N_{1, p, \left|\cdot\right|}\leq 3$ for any $p\geq1$. Moreover, Proposition~\ref{e22prop} and Theorem~\ref{e22thm}
 imply that $N_{1, 2}=2$.  Other quantization based characterizations are developped   in Section~\ref{chardim1}. Then we discuss the completeness of the distance $\mathcal{Q}_{1, 1}$ \big(defined in (\ref{Qnpdisdef})\big) on $\mathcal{P}_{1}(\mathbb{R})$ and of $\mathcal{Q}_{2, 2}$ on $\mathcal{P}_{2}(\mathbb{R})$ with opposite answers  in Section~\ref{42}.

\subsection{Quantization based characterization on $\mathbb{R}$}\label{chardim1}
\begin{prop}[$p=1$]\label{e11d1}
\begin{enumerate}[$(a)$]
\item Let $\mu, \nu\!\in\mathcal{P}_{1}(\mathbb{R})$. If $e_{1,1}(\mu, \cdot)=e_{1,1}(\nu, \cdot)+C$ fror some constant, then $\mu=\nu$ and $C=0$. 

\item If  $\mu_{n}\!\in\mathcal{P}_{1}(\mathbb{R})$, $n\!\in\mathbb{N}^{*}\cup\{\infty\}$, the following properties are equivalent:
\begin{enumerate}[$(i)$]
\item $\mathcal{W}_{1}(\mu_{n},\mu_{\infty})\xrightarrow{\;n\rightarrow+\infty\;}0$,
\item $e_{1,1}(\mu_{n},\cdot)\xrightarrow{\;n\rightarrow+\infty\;}e_{1,1}(\mu_{\infty},\cdot)$ uniformly,
\item $e_{1,1}(\mu_{n},\cdot)\xrightarrow{\;n\rightarrow+\infty\;}e_{1,1}(\mu_{\infty},\cdot)$ pointwise.
\end{enumerate}
\item The distance $\mathcal{Q}_{1,1}$ and $\mathcal{W}_{1}$ are   topologically equivalent on $\mathcal{P}_{1}(\mathbb{R})$ and $N_{1,1}=1$. 
\end{enumerate}
\end{prop}

\begin{proof}$(a)$ The function $e_{1,1}(\mu,\cdot)$ reads $ x\mapsto\int_{\mathbb{R}}\left|\xi-x\right|\mu(d\xi)$, hence it is convex and its right derivative is given by $x\mapsto-1+2\mu\big(]-\infty,x]\big)$. So if $e_{1,1}(\mu,\cdot)=e_{1,1}(\nu,\cdot)+C$, we have $\mu\big(]-\infty,x]\big)=\nu\big(]-\infty,x]\big)$ for all $x\!\in\mathbb{R}$, which implies $\mu=\nu$ (and $C=0$).

\noindent$(b)$ 
It is obvious that $(i)\Rightarrow(ii)$ and $(ii)\Rightarrow(iii)$. Now we prove $(iii)\Rightarrow(i)$. 

For every $n\geq1$, $e_{1,1}(\mu_{n},\cdot)$ can also be written as $a\mapsto\int_{\mathbb{R}}\left|\xi-a\right|\mu_{n}(d\xi)$, which is convex with right derivative  at $a$ given by $-1+2\mu_{n}\big(]-\infty,a]\big)$. Consequently, if $e_{1,1}(\mu_{n},\cdot)$ converges pointwise to $e_{1,1}(\mu_{\infty},\cdot)$ on $\mathbb{R}$, then $\mu_{n}\big(]-\infty,a]\big)$ converges pointwise to $\mu_{\infty}\big(]-\infty,a]\big)$ for all $a\!\in\mathbb{R}$ such that $\mu_{\infty}(\big\{a\big\})=0$ by Lemma~\ref{convex}. This implies $\mu_{n}\xRightarrow{(\mathbb{R})}\mu_{\infty}$. 
The convergence of the first moment follows from $e_{1,1}(\mu_{n},0)\xrightarrow{n\rightarrow+\infty}e_{1,1}(\mu_{\infty},0)$. Hence, we conclude that $\mathcal{W}_{1}(\mu_{n},\mu_{\infty})\xrightarrow{n\rightarrow+\infty}0$ by Theorem~\ref{was1}.

\noindent$(c)$ The claim $(c)$ is a direct result from $(a)$ and $(b)$.
\end{proof}

\begin{prop}[Even integer $p\ge 2$]\label{d1peven} Let  $p$ be an even integer, $p\geq2$.

\vspace{-0.4cm}
\begin{enumerate}[$(a)$]
\item Let  $\mu, \nu\!\in\mathcal{P}_{p}(\mathbb{R})$ such  that $e_{2,p}(\mu, \cdot)=e_{2, p}(\nu, \cdot)+ C $ for some real constant $C$. Then $\mu=\nu$.
\item If $\mu_{n}\!\in\mathcal{P}_{p}(\mathbb{R})$,  $n\!\in\mathbb{N}^{*}\cup\{\infty\}$,  the following properties are equivalent:
\begin{enumerate}[$(i)$]
\item $\mathcal{W}_{p}(\mu_{n},\mu_{\infty})\xrightarrow{\;n\rightarrow+\infty\;}0$,
\item $e_{2,p}(\mu_{n},\cdot)\xrightarrow{\;n\rightarrow+\infty\;}e_{2,p}(\mu_{\infty},\cdot)$ uniformly,
\item $e_{2,p}(\mu_{n},\cdot)\xrightarrow{\;n\rightarrow+\infty\;}e_{2,p}(\mu_{\infty},\cdot)$ pointwise.
\end{enumerate}
\item The distances $\mathcal{Q}_{2,p}$ and  $\mathcal{W}_{p}$ are topologically equivalent on $\mathcal{P}_{p}(\mathbb{R})$  and $N_{1, p}=2$.
\end{enumerate}
\end{prop}

The proof of Proposition~\ref{d1peven} is based on the following lemma. 
\begin{lem}\label{ptop2}
Let $p$ be an even number, $p\geq2$. Let $\mu\!\in\mathcal{P}_{p}(\mathbb{R})$ be  absolutely continuous   with density $f$ \emph{i.e.}   $\mu(d\xi)=f(\xi)d\xi$. If $f$ is continuous, then for any $a, b\!\in\mathbb{R}$ with $a<b$,
\begin{equation}\label{relation}
e_{2,p-2}^{p-2}\big(\mu,(a,b)\big)=\frac{1}{p(p-1)}\left( \frac{\partial^{2} e_{2,p}^{p}}{\partial a^{2}}\big(\mu,(a,b)\big)+\frac{\partial^{2} e_{2,p}^{p}}{\partial b^{2}}\big(\mu,(a,b)\big)-2\frac{\partial^{2} e_{2,p}^{p}}{\partial a\partial b}\big(\mu,(a,b)\big)\right).
\end{equation}
\end{lem}
\begin{proof}[Proof of Lemma~\ref{ptop2}]
Assume that $a<b$, then $e_{2,p}^{p}\big(\mu,(a,b)\big)=\int_{-\infty}^{\frac{a+b}{2}}\left|\xi-a\right|^{p}f(\xi)d\xi+\int_{\frac{a+b}{2}}^{+\infty}\left|\xi-b\right|^{p}f(\xi)d\xi$. Hence, the function $e_{2, p}^{p}\big(\mu, (a, b)\big)$ is continuously differentiable in $a$, since, for any even number $p\geq2$, we have $\frac{\partial \left|\xi-a\right|^{p}f(\xi)}{\partial a}=p(a-\xi)^{p-1}f(\xi)$ and 
\begin{align}
\sup_{a'\in(a-1, a+1)}\left|p(a'-\xi)^{p-1}f(\xi)\right|\leq p2^{p-1}f(\xi)\big[\left|a+1\right|^{p-1}\vee\left|a-1\right|^{p-1}+\left|\xi\right|^{p-1}\big]\!\in L^{1}(\lambda)\nonumber
\end{align}

since $\int_{\mathbb{R}}\left|\xi\right|^{p}f(\xi)d\xi<+\infty$. Likewise, $e_{2, p}^{p}\big(\mu, (a, b)\big)$ is continuously differentiable in $b$ with partial derivatives 
\begin{equation}
\frac{\partial e_{2,p}^{p}\big(\mu,(a,b)\big)}{\partial a}=p\int_{-\infty}^{\frac{a+b}{2}}(a-\xi)^{p-1}f(\xi)d\xi \text{ and }\frac{\partial e_{2,p}^{p}\big(\mu,(a,b)\big)}{\partial b}=p\int_{\frac{a+b}{2}}^{+\infty}(b-\xi)^{p-1}f(\xi)d\xi.\nonumber
\end{equation}

Moreover, we have $\frac{\partial (a-\xi)^{p-1}f(\xi)}{\partial a}=(p-1)(a-\xi)^{p-2}f(\xi)$ and
\begin{align}
\sup_{a'\in(a-1, a+1)}\left|(p-1)(a'-\xi)^{p-2}f(\xi)\right|\leq (p-1)2^{p-2}f(\xi)\big[\left|a+1\right|^{p-2}\vee\left|a-1\right|^{p-2}+\left|\xi\right|^{p-2}\big]\!\in L^{1}(d\xi)\nonumber
\end{align}
since $\int_{\mathbb{R}}\left|\xi\right|^{p}f(\xi)d\xi<+\infty$. By a similar reasoning, one derives that $e_{2, p}^{p}\big(\mu, (a, b)\big)$ is continuously twice differentiable with second order partial derivatives
\begin{align}
\frac{\partial^{2} e_{2,p}^{p}}{\partial a^{2}}\big(\mu,(a,b)\big)&=p\Big[\int_{-\infty}^{\frac{a+b}{2}}(p-1)(a-\xi)^{p-2}f(\xi)d\xi-\frac{1}{2^{p}}(b-a)^{p-1}f(\frac{a+b}{2})\Big],\nonumber\\
\frac{\partial^{2} e_{2,p}^{p}}{\partial b^{2}}\big(\mu,(a,b)\big)&=p\Big[\int_{\frac{a+b}{2}}^{+\infty}(p-1)(b-\xi)^{p-2}f(\xi)d\xi-\frac{1}{2^{p}}(b-a)^{p-1}f(\frac{a+b}{2})\Big],\nonumber\\
\frac{\partial^{2} e_{2,p}^{p}}{\partial a\partial b}\big(\mu,(a,b)\big)&=\frac{\partial^{2} e_{2,p}^{p}}{\partial b\partial a}\big(\mu,(a,b)\big)=-p\frac{1}{2^{p}}(b-a)^{p-1}f\Big(\frac{a+b}{2}\Big).\nonumber
\end{align}

Hence, for every $(a, b)\!\in\mathbb{R}^{2}$ such that $a<b$, 
\begin{equation}
\frac{\partial^{2} e_{2,p}^{p}}{\partial a^{2}}\big(\mu,(a,b)\big)+\frac{\partial^{2} e_{2,p}^{p}}{\partial b^{2}}\big(\mu,(a,b)\big)-2\frac{\partial^{2} e_{2,p}^{p}}{\partial a\partial b}\big(\mu,(a,b)\big)=p(p-1)e_{2,p-2}^{p-2}\big(\mu,(a,b)\big).\hfill\qedhere\nonumber
\end{equation}

\end{proof}

\begin{proof}[Proof of Proposition~\ref{d1peven}]
\noindent$(a)$ \emph{Step 1:  $\mu$ and $\nu$ are absolutely continuous with continuous density functions}. Note that $e_{2,p}(\mu,\cdot)=e_{2,p}(\nu,\cdot)+C $ implies either $\mu=\nu$  by Proposition~\ref{e22prop} if $p=2$,  or, if $p>2$ $e_{2,p-2}(\mu,\cdot)=e_{2,p-2}(\nu,\cdot)$  (after differentiation)  by Lemma~\ref{ptop2}. We can conclude by induction.

\noindent \emph{Step 2 (General case).} 
Let $X$,$Y$ be two random variables with the respective distributions $\mu$ and $\nu$, such that 
\begin{equation}\label{aaaa}
\forall (a,b)\!\in\mathbb{R}^{2},\;\;\;e_{2,p}^{p}\big(X,(a,b)\big)=e_{2,p}^{p}\big(Y,(a,b)\big) +C.
\end{equation}
Let $Z$ be a random variable with probability distribution $\mathbb{P}_{Z}=\mathcal{N}(0,1)$, independent of $X$ and $Y$. For every $\varepsilon>0$, 
\begin{equation}\label{e2pp}
e_{2,p}^{p}\big(X+\varepsilon Z, (a,b)\big)=\iint\min_{x\in\{a,b\}}\left|\xi+\varepsilon z-x\right|^{p}\mu(d\xi)\mathbb{P}_{Z}(dz)=\int e_{2,p}^{p}\big(X, (a,b)-\varepsilon z\big) \mathbb{P}_{Z}(dz).
\end{equation}

\vskip-0.2cm
We derive from (\ref{aaaa}) and (\ref{e2pp}) that
\begin{equation}\label{cccc}
\forall (a,b)\!\in\mathbb{R}^{2},\;\;\;e_{2,p}^{p}\big(X+\varepsilon Z,(a,b)\big)=e_{2,p}^{p}\big(Y+\varepsilon Z,(a,b)\big)+C.
\end{equation}
Moreover, the random variables $X+\varepsilon Z$ and $Y+\varepsilon Z$ have   distributions $\mathcal{N}(0,\varepsilon^{2})\ast\mu$ and $\mathcal{N}(0,\varepsilon^{2})\ast\nu$ respectively, both with continuous densities. It follows from \textit{Step 1} that $\text{Law}(X+\varepsilon Z)=\text{Law}(Y+\varepsilon Z)$ for every $\varepsilon>0$ so that Law($X$)=Law($Y$) by letting $\varepsilon\rightarrow0$.

\vskip 0.15cm \noindent$(b)$ It is obvious that $(i)\Rightarrow(ii)$ and $(ii)\Rightarrow(iii)$. Now we prove $(iii)\Rightarrow(i)$. It follows from Lemma~\ref{ptop2} that $e_{2,p}(\mu_{n},\cdot)\xrightarrow{n\rightarrow+\infty}e_{2,p}(\mu_{\infty},\cdot)$ implies $e_{2,p-2}(\mu_{n},\cdot)\xrightarrow{n\rightarrow+\infty}e_{2,p-2}(\mu_{\infty},\cdot)$ and, by induction, yields $e_{2,2}(\mu_{n},\cdot)\xrightarrow{n\rightarrow+\infty}e_{2,2}(\mu_{\infty},\cdot)$, so that Theorem~\ref{e22thm} and Theorem~\ref{was1} imply that $\mu_{n}$ converges weakly to $\mu_{\infty}$. The convergence of the $p$-th moment follows from $e_{2,p}(\mu_{n},0)\xrightarrow{n\rightarrow+\infty}e_{2,p}(\mu_{\infty},0)$. Hence $\mathcal{W}_{p}(\mu_{n},\mu_{\infty})\xrightarrow{n\rightarrow+\infty}0$ by Theorem~\ref{was1}.

\noindent$(c)$ 
The claim $(a)$ and $(b)$ directly imply that if $p$ is an even integer, $p\geq2$, the distances $\mathcal{Q}_{2,p}$ and  $\mathcal{W}_{p}$ are topologically equivalent on $\mathcal{P}_{p}(\mathbb{R})$ and $N_{1, p}\leq2$. Now we prove that $N_{1, p}=2$. Note that for every $x\in\mathbb{R}$, $e_{1,p}^{p}(\mu, x)=\int_{\mathbb{R}}\left|\xi-x\right|^{p}\mu(d\xi)=\int_{\mathbb{R}}(\xi^{2}-2\xi x+x^{2})^{\frac{p}{2}}\mu(d\xi)$, which is a polynome in $x$ and whose coefficients are the $k$-th moments of $\mu$, $k\in\{1, ..., p\}$. Thus, as soon as two different distributions $\mu$ and $\nu$ have the same first $p$ moments, $e^p_{1,p}(\mu, \cdot) = e^p_{1,p}(\nu, \cdot)$. This implies $N_{1, p}>1$.
\end{proof}

\subsection{About  completeness of  $\big(\mathcal{P}_{1}(\mathbb{R}), \mathcal{Q}_{1,1}\big)$ and  $\big(\mathcal{P}_{2}(\mathbb{R}), \mathcal{Q}_{N, 2}\big)$}\label{42}

We know from~\cite{bolley2008separability} that for $p\geq1$, $(\mathcal{P}_{p}(\mathbb{R}), \mathcal{W}_{p})$ is a complete space and we have proved that $\mathcal{Q}_{1, 1}$ (respectively $Q_{2, 2}$) is topologically equivalent to $\mathcal{W}_{1}$ (resp. $\mathcal{W}_{2}$) on $\mathcal{P}_{1}(\mathbb{R})$ (resp. $\mathcal{P}_{2}(\mathbb{R})$). Now we discuss whether $\mathcal{Q}_{1, 1}$ and $\mathcal{Q}_{2, 2}$ are  complete distances. 

\begin{prop}\label{cauchysequencedim1}
The metric space $\big(\mathcal{P}_{1}(\mathbb{R}), \mathcal{Q}_{1,1}\big)$ is complete.
\end{prop}
\begin{proof}
The inequality (\ref{control}) directly implies that a Cauchy sequence in $\big(\mathcal{P}_{1}(\mathbb{R}), \mathcal{W}_{1}\big)$ is also a Cauchy sequence in $\big(\mathcal{P}_{1}(\mathbb{R}), \mathcal{Q}_{1, 1}\big)$. Now let $(\mu_{n})_{n\geq1}$ be a Cauchy sequence in $\big(\mathcal{P}_{1}(\mathbb{R}), \mathcal{Q}_{1, 1}\big)$. It follows from the definition of $\mathcal{Q}_{1,1}$ that $\big(e_{1,1}(\mu_{n}, \cdot)-e_{1, 1}(\delta_{0}, \cdot)\big)_{n\geq1}$ is a Cauchy sequence in $\big(\mathcal{C}_{b}(\mathbb{R}, \mathbb{R}), \left\Vert\cdot\right\Vert_{\sup}\big)$.

As $\big(\mathcal{C}_{b}(\mathbb{R}, \mathbb{R}), \left\Vert\cdot\right\Vert_{\sup}\big)$ is complete, there exists a function $g\!\in \mathcal{C}_{b}(\mathbb{R}, \mathbb{R}) $ such that 
\begin{equation}\label{cvgsup11}
\left\Vert \big(e_{1,1}(\mu_{n}, \cdot)-e_{1, 1}(\delta_{0}, \cdot)\big)-g\right\Vert_{\sup}\xrightarrow{n\rightarrow+\infty}0.
\end{equation}
Note that for any $a\!\in\mathbb{R}$, $e_{1, 1}(\delta_{0}, a)=\left|a\right|$. 
The sequence $e_{1,1}(\mu_{n}, 0)-e_{1,1}(\delta_{0}, 0)=e_{1,1}(\mu_{n}, 0)$ is also a Cauchy sequence in $\mathbb{R}$. Therefore, $\big(e_{1,1}(\mu_{n}, 0)\big)_{n\geq1}= \big(\int_{\mathbb{R}}\left|\xi\right|\mu_{n}(d\xi)\big)_{n\geq1}$ is bounded, which implies that $(\mu_{n})_{n\geq1}$ is tight. It follows from Prohorov's theorem that there exists a subsequence $(\mu_{\varphi(n)})_{n\geq1}$ weakly converging to $\widetilde{\mu}_{\infty}$. Moreover, by Fatou's lemma in distribution, $\widetilde{\mu}_{\infty}\!\in\mathcal{P}_{1}(\mathbb{R})$ since $\int_{\mathbb{R}}\left|\xi\right|\widetilde{\mu}_{\infty}(d\xi)\leq\liminf_{n}\int_{\mathbb{R}}\left|\xi\right|\mu_{\varphi(n)}(d\xi)<+\infty$.

\smallskip
Now, we prove that $g=e_{1,1}(\widetilde{\mu}, \cdot)-e_{1,1}(\delta_{0}, \cdot)$. 
First, let us define a function $f_{a}(\xi)\coloneqq\left|\xi-a\right|-\left|\xi\right|$. For every $a\!\in\mathbb{R}$, $f_{a}$ is bounded and continuous. Hence, the weak convergence of $(\mu_{\varphi(n)})_{n\geq1}$ implies that $\displaystyle \int_{\mathbb{R}}f_{a}(\xi)\mu_{\varphi(n)}(d\xi)\xrightarrow{n\rightarrow+\infty}\int_{\mathbb{R}}f_{a}(\xi)\widetilde{\mu}_{\infty}(d\xi)$. 

Besides, $\int_{\mathbb{R}}f_{a}(\xi)\mu_{\varphi(n)}(d\xi)=\int_{\mathbb{R}}\big[\left|\xi-a\right|-\left|\xi\right|\big]\mu_{\varphi(n)}(d\xi)=e_{1,1}(\mu_{\varphi(n)}, a)-e_{1,1}(\mu_{\varphi(n)}, 0)$,
which converges to $\big(g(a)+e_{1,1}(\delta_{0}, a)\big)-\big(g(0)+e_{1,1}(\delta_{0}, 0)\big)$ as $n\rightarrow+\infty$ by (\ref{cvgsup11}). Hence, for every $a\!\in\mathbb{R}$,
\[
\big(g(a)+e_{1,1}(\delta_{0}, a)\big)-\big(g(0)+\underset{=0}{\underbrace{e_{1,1}(\delta_{0}, 0}})\big)=\int_{\mathbb{R}}f_{a}(\xi)\widetilde{\mu}_{\infty}(d\xi)=e_{1,1}(\widetilde{\mu}_{\infty}, a)-e_{1,1}(\widetilde{\mu}_{\infty}, 0),\]
i.e. $e_{1,1}(\widetilde{\mu}_{\infty}, a)-e_{1,1}(\delta_{0}, a)-g(a)=e_{1,1}(\widetilde{\mu}_{\infty}, 0)-g(0)$. Setting $C=g(0)-e_{1,1}(\widetilde{\mu}_{\infty}, 0)$, we derive that for every $a\!\in\mathbb{R}$,
\begin{align}\label{gconst}
e_{1,1}(\widetilde{\mu}_{\infty}, a)-e_{1,1}(\delta_{0}, a)-g(a)=C.
\end{align}

Now we prove that $C=0$. 
Generally, for any $\nu\!\in\mathcal{P}_{1}(\mathbb{R})$, one has
\begin{align}
&\lim_{a\rightarrow+\infty}\big(e_{1,1}(\nu, a \big)-e_{1,1}(\delta_{0}, a)\big)=\lim_{a\rightarrow+\infty}\big(e_{1,1}(\nu, a \big)-\left|a\right|\big)
=\lim_{a\rightarrow+\infty}\big(e_{1,1}(\nu, a \big)-a\big)\nonumber\\
&=\lim_{a\rightarrow+\infty}\Big(\int_{\mathbb{R}}\left|\xi-a\right|\nu(d\xi)-a\Big)=\lim_{a\rightarrow+\infty}\Big(\int_{\{\xi\geq a\}}(\xi-a)\nu(d\xi)+\int_{\{\xi<a\}}(a-\xi)\nu(d\xi)-a\Big)\nonumber\\
&=\lim_{a\rightarrow+\infty}\Big(\int_{\{\xi\geq a\}}\xi\nu(d\xi)-2\int_{\{\xi\geq a\}}a\nu(d\xi)+\int_{\{\xi<a\}}(-\xi)\nu(d\xi)\Big).\nonumber
\end{align}

As $\nu\!\in\mathcal{P}_{1}(\mathbb{R})$ i.e. $\int_{\mathbb{R}}\left|\xi\right|\nu(d\xi)<+\infty$, we derive that $\lim_{a\rightarrow+\infty}\int_{\xi<a}(-\xi)\nu(d\xi)=\int_{\mathbb{R}}(-\xi)\nu(d\xi)$ and 
$\lim_{a\rightarrow+\infty}\int_{\{\xi\geq a\}}\xi\nu(d\xi)=0$. This implies 
\begin{align}
0\leq\lim_{a\rightarrow+\infty}\int_{\{\xi\geq a\}}a\,\nu(d\xi)\leq \lim_{a\rightarrow+\infty}\int_{\{\xi\geq a\}}\xi\,\nu(d\xi)=0.\nonumber
\end{align}
After a similar calculation with $\lim_{a\rightarrow-\infty}\big(e_{1,1}(\nu, a \big)-e_{1,1}(\delta_{0}, a)\big)$, we get 
\begin{align}\label{limainfini}
\lim_{a\rightarrow+\infty}\big[e_{1,1}(\nu, a \big)-e_{1,1}(\delta_{0}, a)\big]=\int_{\mathbb{R}}(-\xi)\nu(d\xi)\;\,\text{and}\,\lim_{a\rightarrow-\infty}\big[e_{1,1}(\nu, a \big)-e_{1,1}(\delta_{0}, a)\big]=\int_{\mathbb{R}}\xi\nu(d\xi).
\end{align}
Combining (\ref{gconst}) and (\ref{limainfini}) with $\nu=\widetilde{\mu}_{\infty}$ shows that 
\begin{align}
\lim_{a\rightarrow+\infty}g(a)=-C-\int_{\mathbb{R}}\xi\widetilde{\mu}_{\infty}(d\xi)\;\text{and}\;\lim_{a\rightarrow-\infty}g(a)=-C+\int_{\mathbb{R}}\xi\widetilde{\mu}_{\infty}(d\xi).\nonumber
\end{align}
On the other hand, for every $n\geq1$, (\ref{limainfini}) applied to $\nu=\mu_{\varphi(n)}$ implies 
\begin{align}
\lim_{a\rightarrow\pm\infty}e_{1, 1}(\mu_{\varphi(n)}, a)-e_{1, 1}(\delta_{0}, a)=\mp\int_{\mathbb{R}}\xi\mu_{\varphi(n)}(d\xi).\nonumber
\end{align}
Up to a new extraction of $\mu_{\varphi(n)}$, still denoted by $\mu_{\varphi(n)}$, we may assume that $\int_{\mathbb{R}}\xi\mu_{\varphi(n)}(d\xi)\rightarrow\widetilde{C}\!\in\mathbb{R}$ as $n\rightarrow+\infty$ since $\big(e_{1,1}(\mu_{n}, 0)\big)_{n\geq1}=\big(\int_{\mathbb{R}}\left|\xi\right|\mu_{n}(d\xi)\big)_{n\geq1}$ is bounded. 

Now the uniform convergence (\ref{cvgsup11}) implies that
\begin{align*}
\lim_{n}&\lim_{a\rightarrow\pm\infty}\Big[e_{1,1}(\mu_{\varphi(n)}, a)-e_{1, 1}(\delta_{0}, a)-g(a)\Big]=0 \nonumber
\end{align*}
so that $\widetilde{C}=C+\int_{\mathbb{R}}\xi\widetilde{\mu}_{\infty}(d\xi)=-C+\int_{\mathbb{R}}\xi\widetilde{\mu}_{\infty}(d\xi)$, which in turn implies $C=0$, i.e. $g=e_{1,1}(\widetilde{\mu}_{\infty}, \cdot)-e_{1,1}(\delta_{0}, \cdot)$. Then it follows from (\ref{cvgsup11}) that 
\begin{align}
\left\Vert \big(e_{1,1}(\mu_{n}, \cdot )-e_{1,1}(\delta_{0}, \cdot)\big)-\big(e_{1,1}(\widetilde{\mu}_{\infty}, \cdot)-e_{1,1}(\delta_{0}, \cdot)\big)\right\Vert=\left\Vert e_{1,1}(\mu_{n}, \cdot)-e_{1,1}(\widetilde{\mu}_{\infty}, \cdot)\right\Vert_{\sup}\xrightarrow{n\rightarrow+\infty}0\nonumber
\end{align}

Hence, $\mathcal{W}_{1}(\mu_{n}, \widetilde{\mu}_{\infty})\rightarrow 0$ by applying Proposition~\ref{e11d1}, that is, $(\mu_{n})_{n\geq1}$ is a Cauchy sequence in $(\mathcal{P}_{1}(\mathbb{R}), \mathcal{W}_{1})$. The completeness of $(\mathcal{P}_{1}(\mathbb{R}), \mathcal{W}_{1})$ implies immediately that $(\mathcal{P}_{1}(\mathbb{R}), \mathcal{Q}_{1,1})$ is complete. 
\end{proof}

\begin{thm}\label{notcomplet}
For any $N\geq2$,  the metric space $\big(\mathcal{P}_{2}(\mathbb{R}),\mathcal{Q}_{N,2}\big)$ is not   complete. 
\end{thm}

 \vskip -0.2cm
We will build a  sequence on $\mathcal{P}_{2}(\mathbb{R})$ which is Cauchy for $\mathcal{Q}_{N, 2}$ but not for $\mathcal{W}_{2}$. First, we have the following result.

\begin{lem}\label{cauchyunif}
Let $(\mu_{n})_{n\geq1}$ be a $\mathcal{P}_{2}(\mathbb{R}^{d})$-valued sequence which converges weakly to $\mu_{\infty}$ and, for $n\!\in\mathbb{N}^{*}\cup\{\infty\}$, let $X_{n}$ denote a $\mu_n$-distributed random variable . Assume that $\lim_{n}\mathbb{E}\left|X_{n}\right|^{2}$ exists and is finite. Then 
\begin{equation}\label{unifcauchy}
\sup_{a\in\mathbb{R}^{d}}\left|e_{2,2}\big(\mu_{n}, (a,a)\big)-\sqrt{e_{2,2}^{2}\big(\mu_{\infty}, (a,a)\big)+C_{0}}\right|\xrightarrow{n\rightarrow+\infty}0,
\end{equation}
where $\displaystyle C_{0}=\lim_{n}\mathbb{E}\left|X_{n}\right|^{2}-\mathbb{E}\left|X_{\infty}\right|^{2}\in[0,+\infty)$. 
\end{lem}

\begin{proof}[Proof of Lemma~\ref{cauchyunif}]
An elementary computation shows that
\[e_{2,2}^{2}\big(\mu_{n}, (a,a)\big)=\int_{\mathbb{R}^{d}}\left|\xi-a\right|^{2}\mu_{n}(d\xi)=\int_{\mathbb{R}^{d}}\left|\xi\right|^{2}\mu_{n}(d\xi)-2\Big(\int_{\mathbb{R}^{d}}\xi\mu_{n}(d\xi)\,\big{|}\,a\Big)+\left|a\right|^{2}.
\]

\vskip -0.2cm
As $\Big(\int_{\mathbb{R}^{d}}\left|\xi\right|^{2}\mu_{n}(d\xi)\Big)_{n\geq1}$ is bounded and $\mu_{n}\xRightarrow{\;(\mathbb{R}^{d})\;}\mu_{\infty}$, we have $\int_{\mathbb{R}^{d}}\xi\mu_{n}(d\xi)\rightarrow\int_{\mathbb{R}^{d}} \xi\mu_{\infty}(d\xi)$. It follows that 
\begin{align}
&e_{2,2}^{2}\big(\mu_{n}, (a,a)\big) =\int_{\mathbb{R}^{d}}\left|\xi\right|^{2}\mu_{n}(d\xi)-2\Big(\int_{\mathbb{R}^{d}}\xi\mu_{n}(d\xi)\,\big{|}\, a\Big)+\left|a\right|^{2}\nonumber\\
&\hspace{0.4cm}\xrightarrow{\;n\rightarrow+\infty\;}\int_{\mathbb{R}^{d}}\left|\xi\right|^{2}\mu_{\infty}(d\xi)+C_{0}-2\Big(\int_{\mathbb{R}^{d}}\xi\mu_{\infty}(d\xi)\,\big{|}\,a\Big)+\left|a\right|^{2}=e_{2,2}^{2}\big(\mu_{\infty}, (a,a)\big) + C_{0}\nonumber.
\end{align}

Therefore, for every compact set $K$ in $\mathbb{R}^{d}$, we have 
\begin{equation}\label{c1}
\sup_{a\in K}\left|e_{2,2}\big(\mu_{n}, (a,a)\big)-\sqrt{e_{2,2}^{2}\big(\mu_{\infty}, (a,a)\big)+C_{0}}\right|\xrightarrow{n\rightarrow+\infty}0,
\end{equation}
owing to Arzel\'a-Ascoli theorem, since all  functions $e_{N,p}$ are  1-Lipschitz continuous  (see~\eqref{1lip}).  
On the other hand, we have
\begin{small}\begin{flalign} 
&\hspace{1cm}\left|e_{2,2}\big(\mu_{n}, (a,a)\big)-\sqrt{e_{2,2}^{2}(\mu_{\infty}, (a,a))+C_{0}}\right|&\nonumber\\
&\hspace{2cm}=\frac{\left|e_{2,2}^{2}\big(\mu_{n}, (a,a)\big)-\Big(e_{2,2}^{2}\big(\mu_{\infty}, (a,a)\big)+C_{0}\Big)\right|}{e_{2,2}\big(\mu_{n}, (a,a)\big)+\sqrt{e_{2,2}^{2}\big(\mu_{\infty}, (a,a)\big)+C_{0}}}&\nonumber\\
&\hspace{2cm}=\frac{\left|\mathbb{E}\big(\left|X_{n}\right|^{2}-2(a\,|\,X_{n})+\left|a\right|^{2}\big)-\mathbb{E}\big(\left|X_{n}\right|^{2}-2(a\,|\,X_{n})+\left|a\right|^{2}\big)-C_{0}\right|}{\left\Vert X_{n}-a\right\Vert_{2}+\left\Vert X_{\infty}-a\right\Vert_{2}}&\nonumber\\
&\hspace{2cm}\leq \frac{2\left|(a\mid \mathbb{E}X_{\infty}-\mathbb{E}X_{n})\right|+\left|\mathbb{E}\left|X_{n}\right|^{2}-\mathbb{E}\left|X_{\infty}\right|^{2}-C_{0}\right|}{\left\Vert X_{n}-a\right\Vert_{2}+\left\Vert X_{\infty}-a\right\Vert_{2}}&\nonumber\\
&\hspace{2cm}\leq \frac{2\left|a\right|\left|\mathbb{E}X_{\infty}-\mathbb{E}X_{n}\right|+\left|\mathbb{E}\left|X_{n}\right|^{2}-\mathbb{E}\left|X_{\infty}\right|^{2}-C_{0}\right|}{\big| \left\Vert X_{n}\right\Vert_{2}-\left|a\right|\big|+\big| \left\Vert X_{\infty}\right\Vert_{2}-\left|a\right|\big|}.&
\end{flalign}\end{small}
Let $A\coloneqq 2\sup_{n \in\mathbb{N}\cup\{\infty\}}\mathbb{E}\left|X_{n}\right|^{2}$, then 
\begin{small}\begin{flalign}\label{c2}
&\hspace{1cm}\sup_{\left|a\right|>A}\left|e_{2,2}\big(\mu_{n}, (a,a)\big)-\sqrt{e_{2,2}^{2}(\mu_{\infty}, (a,a))+C_{0}}\right|&\nonumber\\
&\hspace{2cm}\leq \sup_{\left|a\right|>A}\frac{2\left|a\right|\left|\mathbb{E}X_{\infty}-\mathbb{E}X_{n}\right|+\left|\mathbb{E}\left|X_{n}\right|^{2}-\mathbb{E}\left|X_{\infty}\right|^{2}-C_{0}\right|}{\left|a\right|-\left\Vert X_{n}\right\Vert_{2}+\left|a\right|-\left\Vert X_{\infty}\right\Vert_{2}}&\nonumber\\
&\hspace{2cm}\leq\sup_{\left|a\right|>A}\frac{2\left|a\right|\left|\mathbb{E}X_{\infty}-\mathbb{E}X_{n}\right|+\left|\mathbb{E}\left|X_{n}\right|^{2}-\mathbb{E}\left|X_{\infty}\right|^{2}-C_{0}\right|}{2\left|a\right|-A}&\nonumber\\
&\hspace{2cm}\leq \sup_{\left|a\right|>A}2\left|\mathbb{E}X_{\infty}-\mathbb{E}X_{n}\right|+\frac{\left|\mathbb{E}\left|X_{n}\right|^{2}-\mathbb{E}\left|X_{\infty}\right|^{2}-C_{0}\right|}{A} \xrightarrow{n\rightarrow+\infty}0&
\end{flalign}\end{small}
Hence, (\ref{c1}) and (\ref{c2}) imply that \[\sup_{a\in\mathbb{R}^{d}}\left|e_{2,2}\big(\mu_{n}, (a,a)\big)-\sqrt{e_{2,2}^{2}\big(\mu_{\infty}, (a,a)\big)+C_{0}}\right|\xrightarrow{n\rightarrow+\infty}0.\hfill\qedhere\]\end{proof}

Let $Z:\Omega\to \mathbb{R}$ be  $\mathcal{N}(0,1)$-distributed.  We define for every $n\!\in\mathbb{N}$,
\begin{equation}\label{defcontre}
X_{n}\coloneqq e^{\frac{n}{2}Z-\frac{n^{2}}{4}}.
\end{equation}
For $n\geq1$, let $\mu_{n}$ denote the probability distribution of $X_{n}$.  It is obvious that $X_{n}$ converges a.s. to $X_{\infty}=0$, so that $\mu_{\infty}=\delta_{0}$. Moreover, for every $p>0$, $\mathbb{E}\,X_{n}^{p}=e^{\frac{pn^{2}}{8}(p-2)}$. Hence, $\mathbb{E}\,X_{n}=e^{-\frac{n^{2}}{8}}\longrightarrow0=\mathbb{E}\,X_{\infty}$ as $n\rightarrow +\infty$ so that $\mathcal{W}_{1}(\mu_{n}, \mu_{\infty})\rightarrow0$ whereas $\mathbb{E}\,X_{n}^{2}=1$ for every $n\!\in\mathbb{N}$.

Hence $\mathbb{E}X_{n}^{2}$ does not converge to $\mathbb{E}\,X_{\infty}^{2}=0$, which entails that $\mu_{n}$ does not converge to $\mu_{\infty}$ for the Wasserstein distance $\mathcal{W}_{2}$ and thus $\mu_{n}$ is not a $\mathcal{W}_{2}$-Cauchy sequence. We first prove $(\mu_{n})_{n\geq1}$ is a Cauchy sequence in $\big(\mathcal{P}_{2}(\mathbb{R}), \mathcal{Q}_{2, 2}\big)$. The proof relies on the following three lemmas. 

\begin{lem}\label{normc}
Let $Z:\Omega\to \mathbb{R}$ be  $\mathcal{N}(0,1)$-distributed. Then, $\forall z>0$, $\mathbb{P}(Z\geq z)\leq\frac{e^{-\frac{z^{2}}{2}}}{z\sqrt{2\pi}}$.
\end{lem}
\vspace{-0.7cm}
\begin{proof}
$\mathbb{P}(Z\geq z)=\int_{z}^{+\infty}\frac{1}{\sqrt{2\pi}}e^{-\frac{x^{2}}{2}}dx\leq\int_{z}^{+\infty}\frac{x}{z}\frac{1}{\sqrt{2\pi}}e^{-\frac{x^{2}}{2}}dx =\frac{e^{-\frac{z^{2}}{2}}}{z\sqrt{2\pi}}.$
\end{proof}

\begin{lem} Define $(X_{n})$ as in (\ref{defcontre}), then $\sup_{K\geq0} K\,\mathbb{E}(X_{n}-K)_{+}\rightarrow0$ as $n\rightarrow+\infty$.
\end{lem}
\begin{proof} We have 
\begin{align*}
&K\,\mathbb{E}(X_{n}-K)_{+}=K\int_{0}^{\infty}\mathbb{P}\Big((X_{n}-K)_{+}\geq u\Big)du=K\int_{0}^{+\infty}\mathbb{P}(X_{n}>u+K)du\nonumber\\
&\hspace{0.3cm}=K\int_{K}^{+\infty}\mathbb{P}(X_{n}\geq v)dv=K\int_{K}^{+\infty}\mathbb{P}\Big(e^{\frac{n}{2}Z-\frac{n^{2}}{4}}\geq v\Big)dv\\
&\hspace{0.3cm}=K\int_{K}^{+\infty}\mathbb{P}\Big(Z\geq\frac{n}{2}+\frac{2}{n}\ln v\Big)dv=K\int_{\ln K}^{\infty}\mathbb{P}\Big(Z\geq\frac{n}{2}+\frac{2}{n}u\Big)e^{u}du \;\;\text{(setting $u=\ln v$)}.\nonumber
\end{align*}
By Lemma~\ref{normc}, 
$\mathbb{P}\Big(Z\geq\frac{n}{2}+\frac{2}{n}u\Big)\leq \frac{1}{\sqrt{2\pi}}\frac{e^{-\frac{1}{2}(\frac{n}{2}+\frac{2}{n}u)^{2}}}{\frac{n}{2}+\frac{2}{n}u}=\frac{1}{\sqrt{2\pi}}\frac{e^{-\frac{n^{2}}{8}-\frac{2}{n^{2}}u^{2}-u}}{\frac{n}{2}+\frac{2}{n}u}.$
It follows that,
\begin{align}\label{KK}
&K\,\mathbb{E}(X_{n}-K)_{+}\leq K\int_{\ln K}^{\infty}\frac{e^{-\frac{n^{2}}{8}-\frac{2}{n^{2}}u^{2}-u}}{\frac{n}{2}+\frac{2}{n}u}e^{u}\frac{du}{\sqrt{2\pi}}\leq \frac{Ke^{-\frac{n^{2}}{8}}}{\frac{n}{2}+\frac{2}{n}\ln K}\int_{\ln K}^{\infty}e^{-\frac{2}{n^{2}}u^{2}}\frac{du}{\sqrt{2\pi}}\nonumber\\
&\hspace{0.3cm}=\frac{Ke^{-\frac{n^{2}}{8}}}{\frac{n}{2}+\frac{2}{n}\ln K}\int_{\frac{2}{n}\ln K}^{\infty}e^{-\frac{w^{2}}{2}}\frac{n}{2}\frac{dw}{\sqrt{2\pi}}\;\;\text{(by setting $w=\frac{2}{n}u$)}\nonumber\\
&\hspace{0.3cm}=\frac{Ke^{-\frac{n^{2}}{8}}}{\frac{n}{2}+\frac{2}{n}\ln K}\frac{n}{2}\;\mathbb{P}\Big(Z\geq \frac{2}{n}\ln K\Big)\leq \frac{nKe^{-\frac{n^{2}}{8}}}{2(\frac{n}{2}+\frac{2}{n}\ln K)}\frac{e^{-\frac{1}{2}\frac{4}{n^{2}}(\ln K)^{2}}}{\sqrt{2\pi}\frac{2}{n}\ln K}\;\;\text{(by Lemma~\ref{normc})}\nonumber\\
&\hspace{0.3cm}=\frac{n}{2\sqrt{2\pi}}e^{-\frac{n^{2}}{8}}\frac{K e^{-\frac{2}{n^{2}}(\ln K)^{2}}}{(1+\frac{4}{n^{2}}\ln K)\ln K}=\frac{n}{2\sqrt{2\pi}}e^{-\frac{n^{2}}{8}}\frac{ e^{\ln K(1-\frac{2}{n^{2}}\ln K)}}{(1+\frac{4}{n^{2}}\ln K)\ln K}.
\end{align}
Since the function $u\mapsto u(1-\frac{2}{n^{2}}u)$ attains its maximum at $u=\frac{n^{2}}{4}$ with  maximum value  $\frac{n^{2}}{8}$, we will discuss the value of $K\,\mathbb{E}(X_{n}-K)_{+}$ in the following three cases:
\begin{center}
\noindent $(i)$ $K\geq e^{\frac{n^{2}}{4}}$,
$\qquad (ii)$ $e^{\rho\frac{n^{2}}{4}}\leq K\leq e^{\frac{n^{2}}{4}}$,
$\qquad (iii)$ $0\leq K\leq e^{\rho \frac{n^{2}}{4}}$ ,
\end{center}
\noindent with the same fixed  $\rho\in(0,\frac{1}{2})$  in both $(ii)$ and $(iii)$. 

\smallskip
\noindent \textbf{Case (i)}: $K\geq e^{\frac{n^{2}}{4}}$, then $\ln K\geq \frac{n^{2}}{4}$. It follows that 
\begin{align}
K\,\mathbb{E}(X_{n}-K)_{+}&\leq\frac{n}{2\sqrt{2\pi}}e^{-\frac{n^{2}}{8}}\frac{ e^{\ln K(1-\frac{2}{n^{2}}\ln K)}}{(1+\frac{4}{n^{2}}\ln K)\ln K}\leq\frac{n}{2\sqrt{2\pi}}e^{-\frac{n^{2}}{8}}\frac{ e^{\frac{n^{2}}{8}}}{(1+\frac{4}{n^{2}}\times\frac{n^{2}}{4})\frac{n^{2}}{4}}=\frac{1}{n\sqrt{2\pi}}\rightarrow0.\nonumber
\end{align}

\noindent\textbf{Case (ii)}: $e^{\rho\frac{n^{2}}{4}}\leq K\leq e^{\frac{n^{2}}{4}}$ with a fixed $\rho\in(0,\frac{1}{2})$, then $\rho\frac{n^{2}}{4}\leq \ln K\leq \frac{n^{2}}{4}$. It follows that
\begin{align}
K\,\mathbb{E}(X_{n}-K)_{+}&\leq\frac{ne^{-\frac{n^{2}}{8}}}{2\sqrt{2\pi}}\frac{ e^{\ln K(1-\frac{2}{n^{2}}\ln K)}}{(1+\frac{4}{n^{2}}\ln K)\ln K}\leq\frac{ne^{-\frac{n^{2}}{8}}}{2\sqrt{2\pi}}\frac{ e^{\frac{n^{2}}{8}}}{(1+\frac{4}{n^{2}}\times\rho\frac{n^{2}}{4})\rho\frac{n^{2}}{4}}=\frac{2}{n(1+\rho)\rho\sqrt{2\pi}}\rightarrow0.\nonumber
\end{align}

\noindent \textbf{Case (iii)}: $0\leq K\leq e^{\rho \frac{n^{2}}{4}}$ with the same $\rho\in(0,\frac{1}{2})$ as in the situation (ii), then 
\[
K\,\mathbb{E}(X_{n}-K)_{+}\leq e^{\frac{\rho}{4}n^{2}}\mathbb{E}X_{n}=e^{\frac{\rho}{4}n^{2}}\cdot e^{-\frac{n^{2}}{8}}=e^{\frac 14(\rho-\frac{1}{2})n^{2}}\xrightarrow{n\rightarrow+\infty}0.
\]
Therefore, $\sup_{K>0} \;K\,\mathbb{E}(X_{n}-K)_{+}\xrightarrow{n\rightarrow+\infty}0$.
\end{proof}

By Lemma~\ref{cauchyunif}, $\sup_{a\in\mathbb{R}^{d}}\left|e_{2,2}\big(\mu_{n}, (a,a)\big)-\sqrt{e_{2,2}^{2}\big(\mu_{\infty}, (a,a)\big)+C_{0}}\right|\xrightarrow{n\rightarrow+\infty}0$. Consequently, it is reasonable to guess that $e_{N,2}(\mu_{n}, \cdot)\xrightarrow[\;\;n\rightarrow+\infty\;\;]{\left\Vert\cdot\right\Vert_{\sup}}\sqrt{e_{N,2}^{2}(\mu_{\infty}, \cdot)+1}$ so that $(\mu_{n})_{n\in\mathbb{N}}$ is a Cauchy sequence in $(\mathcal{P}_{2}(\mathbb{R}^{d}), \mathcal{Q}_{N,2})$.
Let $g_{N}: \mathbb{R}^N\to \mathbb{R}_{+}$ be  defined   by
\begin{equation}
(a_{1}, \dots, a_{N})  \mapsto g_{N}\big((a_{1}, \dots, a_{N})\big)\coloneqq \sqrt{e_{N,2}^{2}\big(\mu_{\infty}, (a_{1}, \dots, a_{N})\big)+1}=\sqrt{\min_{1\leq i\leq N}\left|a_{i}\right|^{2}+1}.\nonumber
\end{equation}

\begin{prop}\label{contrelevelN}
For  every $N\geq2$,
\[\sup_{(a_{1},\dots, a_{N})\in\mathbb{R}^{N}}\big{|}e_{N,2}\big(\mu_{n}, (a_{1},\dots, a_{N})\big)-g_{N}\big((a_{1},\dots, a_{N})\big)\big{|}\xrightarrow{n\rightarrow+\infty}0.\] Therefore, $(\mu_{n})_{n\in\mathbb{N}}$ is a Cauchy sequence in $(\mathcal{P}_{2}(\mathbb{R}), \mathcal{Q}_{N,2})$ by the definition of $\mathcal{Q}_{N,2}$.
\end{prop}

\begin{proof} We proceed  by induction. 

\noindent \underline{$\rhd$ $N=2$.} Since the functions $g_{2}$ and $e_{2,2}(\mu_{n}, \cdot)$ are symmetric, it is only necessary to show that 
$\sup_{(a,b)\in\mathbb{R}^{2},\,\left|a\right|\leq\left|b\right|}\big|e_{2,2}\big(\mu_{n}, (a,b)\big)-g_{2}(a,b)\big|\xrightarrow{n\rightarrow+\infty}0$.
Note that when $\left|a\right|\leq\left|b\right|$, $g_{2}(a,b)=\sqrt{\left|a\right|^{2}+1}=g_{2}(a,a)$. We discuss now the value of $\big|e_{2,2}\big(\mu_{n}, (a,b)\big)-g_{2}(a,b)\big|$ in the following four cases, 

$(i)$ $0\leq a\leq b$,\hspace{1cm}$(ii)$ $a\leq 0\leq b$, 
	$\displaystyle \left\{
	\begin{array}{ll}
			(ii,\alpha) & a\leq 0\leq b \hbox{ with }\left|a\right|\leq\frac{1}{2}\left|b\right|\\
		 	(ii,\beta) &  a\leq 0\leq b \hbox{ with } \frac{1}{2}\left|b\right|\leq\left|a\right|\leq\left|b\right|,
	\end{array}\right.$

$(iii)$ $b\leq 0\leq a$, with $\left|a\right|\leq\left|b\right|$, \hspace{1cm} $(iv)$ $b\leq a\leq 0$.

\smallskip

\noindent \textbf{Cases $(iii)$ and $(iv)$: $\mathbf{b<0}$ and $\mathbf{\frac{a+b}{2}<0}$. } The random variables $X_{n}$ are  positive so that $\left|x-a\right|\leq\left|x-b\right|$. Hence $e_{2,2}\big(\mu_{n}, (a,b)\big)=e_{2,2}\big(\mu_{n}, (a,a)\big)$. With a slight abuse of notation, we will write in what follows $(a,b)\in(iii)$ for  $(a,b)\!\in\{(a,b)\!\in\mathbb{R}^{2}\mid b\leq 0\leq a, \text{and} \left|a\right|\leq\left|b\right|\}$. We will adopt the same notation for other cases too. Then for the case $(iii)$ and $(iv)$, it is obvious  by applying Lemma~\ref{cauchyunif} that
\[\sup_{(a,b)\in(iii)\cup(iv)}\big|e_{2,2}\big(\mu_{n}, (a,b)\big)-g_{2}(a,b)\big|=\sup_{(a,b)\in(iii)\cup(iv)}\big|e_{2,2}\big(\mu_{n}, (a,a)\big)-g_{2}(a,a)\big|\xrightarrow{n\rightarrow+\infty}0.
\]
 
\noindent \textbf{Case $(i)$: $\mathbf{0\leq a\leq b}$.} We have
\begin{flalign}
&\sup_{(a,b)\in (i)}\big|e_{2,2}\big(\mu_{n}, (a,b)\big)-g_{2}(a,b)\big|&&\nonumber\\
&\hspace{0.4cm}\leq \;\sup_{(a,b)\in(i)}\big|e_{2,2}\big(\mu_{n}, (a,b)\big)-e_{2,2}\big(\mu_{n}, (a,a)\big)\big|+\big|e_{2,2}\big(\mu_{n}, (a,a)\big)-g_{2}(a,a)\big|&&\nonumber\\
&\hspace{0.4cm}\leq \;\sup_{(a,b)\in(i)}\left|\sqrt{\int_{\mathbb{R}}\left|\xi-a\right|^{2}\wedge\left|\xi-b\right|^{2}\mu_{n}(d\xi)}-\sqrt{\int_{\mathbb{R}}\left|\xi-a\right|^{2}\mu_{n}(d\xi)}\right|+\big|e_{2,2}\big(\mu_{n}, (a,a)\big)-g_{2}(a,a)\big|&&\nonumber\\
&\hspace{0.4cm}\leq \;\sup_{(a,b)\in(i)}\sqrt{\int_{\mathbb{R}}\Big[\left|\xi-a\right|^{2}-\big(\left|\xi-a\right|^{2}\wedge\left|\xi-b\right|^{2}\big)\Big]\mu_{n}(d\xi)}+\big|e_{2,2}\big(\mu_{n}, (a,a)\big)-g_{2}(a,a)\big|\;\:&&\nonumber\\
&\hspace{0.4cm}\;\text{  \small(since $\left|\sqrt{\alpha}-\sqrt{\beta}\right|\leq\sqrt{\beta-\alpha}$ for $\beta>\alpha>0$)}&&\nonumber\\
&\hspace{0.4cm}\leq\;\sup_{(a,b)\in(i)}\sqrt{\int_{\mathbb{R}}\Big(\left|\xi-a\right|^{2}-\left|\xi-b\right|^{2}\Big)_{+}\mu_{n}(d\xi)}+\big|e_{2,2}\big(\mu_{n}, (a,a)\big)-g_{2}(a,a)\big|&&\nonumber\\
&\hspace{0.4cm}\leq\;\sup_{(a,b)\in(i)}\sqrt{\int_{\mathbb{R}}2(b-a)\Big(\xi-\frac{b+a}{2}\Big)_{+}\mu_{n}(d\xi)}+\big|e_{2,2}\big(\mu_{n}, (a,a)\big)-g_{2}(a,a)\big|&&\nonumber\\
&\hspace{0.4cm}\leq\;\sup_{(a,b)\in(i)}2\sqrt{\int_{\mathbb{R}}\frac{b}{2}\Big(\xi-\frac{b}{2}\Big)_{+}\mu_{n}(d\xi)}+\big|e_{2,2}\big(\mu_{n}, (a,a)\big)-g_{2}(a,a)\big|&&\nonumber\\
&\hspace{0.4cm}\leq\;2\sqrt{\sup_{K\geq0} K\,\mathbb{E}(X_{n}-K)_{+}}+\sup_{a\in\mathbb{R}}\big|e_{2,2}\big(\mu_{n}, (a,a)\big)-g_{2}(a,a)\big|\xrightarrow{n\rightarrow+\infty}0.&&\nonumber
\end{flalign}

\item\textbf{Case (ii,$\alpha$): $\mathbf{a\leq 0\leq b}$, with $\mathbf{\left|a\right|\leq\frac{1}{2}\left|b\right|}$.}  We have
\begin{flalign}
\sup_{(a,b)\in (ii,\alpha)}&\big|e_{2,2}\big(\mu_{n}, (a,b)\big)-g_{2}(a,b)\big|&&\nonumber\\
&\leq \;\sup_{(a,b)\in(ii,\alpha)}\big|e_{2,2}\big(\mu_{n}, (a,b)\big)-e_{2,2}\big(\mu_{n}, (a,a)\big)\big|+\big|e_{2,2}\big(\mu_{n}, (a,a)\big)-g_{2}(a,a)\big|&&\nonumber\\
&\leq\;\sup_{(a,b)\in(ii,\alpha)}\sqrt{\int_{\mathbb{R}}2(b-a)\Big(\xi-\frac{b+a}{2}\Big)_{+}\mu_{n}(d\xi)}+\big|e_{2,2}\big(\mu_{n}, (a,a)\big)-g_{2}(a,a)\big|&&\nonumber\\
&\leq\;\sup_{(a,b)\in(ii,\alpha)}\sqrt{\int_{\mathbb{R}}3\cdot b\Big(\xi-\frac{b}{4}\Big)_{+}\mu_{n}(d\xi)}+\big|e_{2,2}\big(\mu_{n}, (a,a)\big)-g_{2}(a,a)\big|&&\nonumber\\
&\leq\;2\sqrt{3}\cdot\sqrt{\sup_{K\geq0} K\,\mathbb{E}(X_{n}-K)_{+}}+\sup_{a\in\mathbb{R}}\big|e_{2,2}\big(\mu_{n}, (a,a)\big)-g_{2}(a,a)\big|\xrightarrow{n\rightarrow+\infty}0.&&\nonumber
\end{flalign}

\vskip-0.2cm
\noindent \textbf{Case (ii,$\beta$): $\mathbf{a\leq 0\leq b}$, with $\mathbf{\frac{1}{2}\left|b\right|\leq\left|a\right|\leq\left|b\right|}$.} One has
\begin{flalign}
\;\sup_{(a,b)\in (ii,\beta)}&\big|e_{2,2}\big(\mu_{n}, (a,b)\big)-g_{2}(a,b)\big|&&\nonumber\\
&\leq \;\sup_{(a,b)\in(ii,\beta)}\big|e_{2,2}\big(\mu_{n}, (a,b)\big)-e_{2,2}\big(\mu_{n}, (a,a)\big)\big|+\big|e_{2,2}\big(\mu_{n}, (a,a)\big)-g_{2}(a,a)\big|&&\nonumber\\
&\leq\;\sup_{(a,b)\in(ii,\beta)}\frac{\big|e_{2,2}^{2}\big(\mu_{n}, (a,b)\big)-e_{2,2}^{2}\big(\mu_{n}, (a,a)\big)\big|}{e_{2,2}\big(\mu_{n}, (a,b)\big)+e_{2,2}\big(\mu_{n}, (a,a)\big)}+\big|e_{2,2}\big(\mu_{n}, (a,a)\big)-g_{2}(a,a)\big|&&\nonumber\\
&\leq\;\sup_{(a,b)\in(ii,\beta)}\frac{\int_{\mathbb{R}}2(b-a)\big(\xi-\frac{b+a}{2}\big)_{+}\mu_{n}(d\xi)}{\left\Vert X_{n}-a\right\Vert_{2}}+\sup_{a\in\mathbb{R}}\big|e_{2,2}\big(\mu_{n}, (a,a)\big)-g_{2}(a,a)\big|&&\nonumber\\
&\leq\;\sup_{(a,b)\in(ii,\beta)}\frac{2(b-a)\,\mathbb{E}\,\big(X_{n}-\frac{b+a}{2}\big)_{+}}{\left\Vert X_{n}-a\right\Vert_{2}}+\sup_{a\in\mathbb{R}}\big|e_{2,2}\big(\mu_{n}, (a,a)\big)-g_{2}(a,a)\big|.&&\nonumber
\end{flalign}
As $\left\Vert X_{n}-a\right\Vert_{2}=\big(\underset{=1}{\underbrace{\mathbb{E}X_{n}^{2}}}\;\underset{\geq0}{\underbrace{-2a\,\mathbb{E}X_{n}}}+\left|a\right|^{2}\big)^{1/2} \geq \sqrt{1+\left|a\right|^{2}}$, we have 
\begin{flalign}\label{situ32}
\sup_{(a,b)\in (ii,\beta)}&\big|e_{2,2}\big(\mu_{n}, (a,b)\big)-g_{2}(a,b)\big|&&\nonumber\\
&\leq\;\sup_{(a,b)\in(ii,\beta)}\frac{2(b+\left|a\right|)\mathbb{E}\big[X_{n}-\frac{b+a}{2}\big]_{+}}{\sqrt{1+\left|a\right|^{2}}}+\sup_{a\in\mathbb{R}}\big|e_{2,2}\big(\mu_{n}, (a,a)\big)-g_{2}(a,a)\big|&&\nonumber\\
&\leq\;\sup_{(a,b)\in(ii,\beta)}\frac{4b\,\mathbb{E}X_{n}}{\sqrt{1+\frac{b^{2}}{4}}}+\sup_{a\in\mathbb{R}}\big|e_{2,2}\big(\mu_{n}, (a,a)\big)-g_{2}(a,a)\big|.&&\nonumber\\
&\leq\;8\,\mathbb{E}X_{n}+\sup_{a\in\mathbb{R}}\big|e_{2,2}\big(\mu_{n}, (a,a)\big)-g_{2}(a,a)\big|\xrightarrow{n\rightarrow+\infty}0.&&\nonumber
\end{flalign}

\vspace{-0.3cm}
\noindent \underline{$\rhd$ \textbf{From $N$ to $N\!+\!1$.}} Assume now that 
{\small
$\displaystyle \sup_{(a_{1},\dots, a_{N})\in\mathbb{R}^{N}}\!\!\!\!\!\!\big|e_{N,2}\big(\mu_{n}, (a_{1},\dots, a_{N})\big)-g_{N}(a_{1},\dots, a_{N})\big|$
}
 goes $0$ as $n\to +\infty$. Then, for the level $N+1$, we assume without loss of generality  that $\left|a_{1}\right|\leq\left|a_{2}\right|\leq\dots\leq\left|a_{N+1}\right|$ since $g_{N+1}$ and $e_{N,2}(\mu_{n}, \cdot)$ are symmetric. Under this assumption, 
{\small
\begin{equation}\label{gn}
g_{N+1}(a_{1},\dots, a_{N+1})=g_{2}(a_{1}, a_{1})=\sqrt{\left|a_{1}\right|^{2}+1}.
\end{equation}
}
 We discuss now the value of 
$\displaystyle\sup_{(a_{1},\dots, a_{N+1})\in\mathbb{R}^{N+1}}\!\!\!\!\!\!\!\!\!\!\! \big|e_{N+1,2}\big(\mu_{n}, (a_{1},\dots, a_{N+1})\big)-g_{N+1}(a_{1},\dots, a_{N+1})\big|$ 
in the following cases:

\smallskip
\small
\noindent $(i)$ 
 $\exists \,i \!\in\{2, \dots, N+1\}$ such that $a_{i}<0$,\hspace{1cm}$(ii)$ $0\leq a_{1}\leq a_{2} \leq \dots \leq a_{N+1}$,

	\smallskip
\noindent $(iii)$ $a_{1}\leq 0 \leq a_{2} \leq \dots \leq a_{N+1}$,$\displaystyle\left\{\begin{array}{ll} (iii,\alpha)&\mbox{\hspace{-0.2cm}$a_{1}\leq 0 \leq a_{2} \leq \dots \leq a_{N+1}$, with $\left|a_{1}\right|\leq\frac{1}{2}\left|a_{N+1}\right|$}\\
		 (iii,\beta) &\mbox{\hspace{-0.2cm}$a_{1}\leq 0 \leq a_{2} \leq \dots \leq a_{N+1}$, with $\frac{1}{2}\left|a_{N+1}\right|\leq\left|a_{1}\right|\leq\left|a_{N+1}\right|$}
		 \end{array}\right.$.
	
\smallskip
\noindent 
\normalsize\textbf{Case $(i)$:} $\mathbf{\exists\,i \!\in\{2, \dots, N+1\}\; \textbf{ such that }\; a_{i}<0}.$
For every $n\geq1$, $X_{n}$ is a.s. positive. Hence, $\left|X_{n}-a_{1}\right|\leq \left|X_{n}-a_{i}\right|$ a.s. since we assume that $\left|a_{1}\right|\leq\left|a_{2}\right|\leq\dots\leq\left|a_{N+1}\right|$. Therefore, \[e_{N+1,2}\big(\mu_{n}, (a_{1},\dots, a_{N+1})\big)=e_{N,2}\big(\mu_{n}, (a_{1},\dots,a_{i-1}, a_{i+1},\dots, a_{N+1})\big).\] It follows from (\ref{gn}) that 
\begin{align}
&\sup_{(a_{1},\dots, a_{N+1}) \in\mathbb{R}^{N+1}}\big|e_{N+1,2}\big(\mu_{n}, (a_{1},\dots, a_{N+1})\big)-g_{N+1}(a_{1},\dots, a_{N+1})\big|\nonumber\\
&=\sup_{(a_{1},\dots,a_{i-1}, a_{i+1},\dots, a_{N+1}) \in\mathbb{R}^{N}}\big|e_{N,2}\big(\mu_{n}, (a_{1},\dots,a_{i-1}, a_{i+1},\dots, a_{N+1})\big)-g_{N}(a_{1},\dots,a_{i-1}, a_{i+1},\dots, a_{N+1})\big|,\nonumber
\end{align}
which converges to 0 as $n\rightarrow +\infty$ owing to the assumption on the level $N$.

\noindent \textbf{Case $(ii)$:} $\mathbf{0\leq a_{1}\leq a_{2} \leq \dots \leq a_{N+1}}.$
\begin{align}\label{situiileveln}
\sup_{0\leq a_{1}\leq a_{2} \leq \dots \leq a_{N+1}}& \big|e_{N+1,2}\big(\mu_{n}, (a_{1},\dots, a_{N+1})\big)-g_{N+1}(a_{1},\dots, a_{N+1})\big|\nonumber\\
\leq &\sup_{0\leq a_{1}\leq a_{2} \leq \dots \leq a_{N+1}}\big|e_{N+1,2}\big(\mu_{n}, (a_{1},\dots, a_{N+1})\big)-e_{N,2}\big(\mu_{n}, (a_{1},\dots, a_{N})\big)\big|\nonumber\\
&+\sup_{0\leq a_{1}\leq a_{2} \leq \dots \leq a_{N+1}}\big|e_{N,2}\big(\mu_{n}, (a_{1},\dots, a_{N})\big)-g_{N}(a_{1},\dots, a_{N})\big|.
\end{align}
The second term on the right hand side of (\ref{situiileveln}) 
converges to 0 as $n\rightarrow +\infty$ owing to the assumption on the level $N$. 

For the first term on the right hand side of (\ref{situiileveln}), we have
\begin{align}
&\sup_{0\leq a_{1}\leq a_{2} \leq \dots \leq a_{N+1}}\big|e_{N+1,2}\big(\mu_{n}, (a_{1},\dots, a_{N+1})\big)-e_{N,2}\big(\mu_{n}, (a_{1},\dots, a_{N})\big)\big|\nonumber\\
&\hspace{0.3cm}=\sup_{0\leq a_{1}\leq a_{2} \leq \dots \leq a_{N+1}}\sqrt{\int_{\mathbb{R}}\min_{1\le i\le N}\left|\xi-a_{i}\right|^{2}\mu_{n}(d\xi)}-\sqrt{\int_{\mathbb{R}}\Big[\min_{1\le i\le N}\left|\xi-a_{}\right|^{2}\Big]\wedge\left|\xi-a_{N+1}\right|^{2}\mu_{n}(d\xi)}\nonumber\\
&\hspace{0.3cm}\leq\sup_{0\leq a_{1}\leq a_{2} \leq \dots \leq a_{N+1}}\sqrt{\int_{\mathbb{R}}\min_{1\le i\le N}\left|\xi-a_{i}\right|^{2}-\Big[\min_{1\le i\le N}\left|\xi-a_{i}\right|^{2}\Big]\wedge\left|\xi-a_{N+1}\right|^{2}\mu_{n}(d\xi)}\nonumber\\
&\hspace{0.3cm}=\sup_{0\leq a_{1}\leq a_{2} \leq \dots \leq a_{N+1}}\sqrt{\int_{\mathbb{R}}\big( \min_{1\le i\le N}\left|\xi-a_{i}\right|^{2}- \left|\xi-a_{N+1}\right|^{2}\big)_{+}\mu_{n}(d\xi)}\nonumber\\
&\hspace{0.3cm}\leq\sup_{0\leq a_{1}\leq a_{2} \leq \dots \leq a_{N+1}}\sqrt{\int_{\mathbb{R}}\big(\left|\xi-a_{1}\right|^{2}-\left|\xi-a_{N+1}\right|^{2}\big)_{+}\mu_{n}(d\xi)}\nonumber\\
&\hspace{0.3cm}=\sup_{0\leq a_{1}\leq a_{2} \leq \dots \leq a_{N+1}}\sqrt{\int_{\mathbb{R}}2(a_{N+1}-a_{1})\big(\xi-\frac{a_{1}+a_{N+1}}{2}\big)_{+}\mu_{n}(d\xi)}\nonumber\\
&\hspace{0.3cm}\leq\sup_{0\leq a_{1}\leq a_{2} \leq \dots \leq a_{N+1}}\sqrt{\int_{\mathbb{R}}2\cdot a_{N+1}\big(\xi-\frac{a_{N+1}}{2}\big)_{+}\mu_{n}(d\xi)}\leq\;2\cdot\sqrt{\sup_{K\geq0} K\,\mathbb{E}(X_{n}-K)_{+}}\xrightarrow{n\rightarrow+\infty}0. \nonumber
\end{align}

\noindent \textbf{Case $(iii,\alpha)$:} $\mathbf{a_{1}\leq 0 \leq a_{2} \leq \dots \leq a_{N+1} \;\textbf{ with }\; \left|a_{1}\right|\leq\frac{1}{2}\left|a_{N+1}\right|}.$
\begin{align}\label{situ3aleveln}
&\sup_{(a_{1},\dots, a_{N+1})\in(iii,\alpha)}\big|e_{N+1,2}\big(\mu_{n}, (a_{1},\dots, a_{N+1})\big)-g_{N+1}(a_{1},\dots, a_{N+1})\big|\nonumber\\
\leq &\sup_{(a_{1},\dots, a_{N+1})\in(iii,\alpha)}\big|e_{N+1,2}\big(\mu_{n}, (a_{1},\dots, a_{N+1})\big)-e_{N,2}\big(\mu_{n}, (a_{1},\dots, a_{N})\big)\big|\nonumber\\
&\qquad +\sup_{(a_{1},\dots, a_{N+1})\in(iii,\alpha)}\big|e_{N,2}\big(\mu_{n}, (a_{1},\dots, a_{N})\big)-g_{N}(a_{1},\dots, a_{N})\big|.
\end{align}
Like in Case $(ii)$, the second term on the right hand side of (\ref{situ3aleveln}) 
converges to 0 as $n\rightarrow +\infty$. 
For the first term of the right hand side of (\ref{situ3aleveln}), we have 
\begin{align}
&\sup_{(a_{1},\dots, a_{N+1})\in(iii,\alpha)} \big|e_{N+1,2}\big(\mu_{n}, (a_{1},\dots, a_{N+1})\big)-e_{N,2}\big(\mu_{n}, (a_{1},\dots, a_{N})\big)\big|\nonumber\\
&\hspace{0.3cm}\leq\sup_{(a_{1},\dots, a_{N+1})\in(iii,\alpha)}\sqrt{\int_{\mathbb{R}}2(a_{N+1}-a_{1})\big(\xi-\frac{a_{1}+a_{N+1}}{2}\big)_{+}\mu_{n}(d\xi)}\nonumber\\
&\hspace{0.3cm}\leq\sup_{(a_{1},\dots, a_{N+1})\in(iii,\alpha)}\sqrt{\int_{\mathbb{R}}3\cdot a_{N+1}\big(\xi-\frac{a_{N+1}}{4}\big)_{+}\mu_{n}(d\xi)}\leq2\sqrt{3}\cdot\sqrt{\sup_{K\geq0} K\,\mathbb{E}(X_{n}-K)_{+}}\longrightarrow0. \nonumber
\end{align}

\noindent\textbf{Case $(iii,\beta)$:} $\mathbf{a_{1}\leq 0 \leq a_{2} \leq \dots \leq a_{N+1} \;\textbf{ with }\; \frac{1}{2}\left|a_{N+1}\right|\leq\left|a_{1}\right|\leq\left|a_{N+1}\right|}.$

Since we assume  $\left|a_{1}\right|\leq\left|a_{2}\right|\leq\dots\leq\left|a_{N+1}\right|$, then for any $i\!\in\{2, \dots, N+1\}$, we have
$\frac{1}{2}\left|a_{i}\right|\leq\left|a_{1}\right|\leq\left|a_{i}\right|$.
It follows that 
\begin{align}\label{situ3bleveln}
\sup_{(a_{1},\dots, a_{N+1})\in(iii,\beta)} & \big|e_{N+1,2}\big(\mu_{n}, (a_{1},\dots, a_{N+1})\big)-g_{N+1}(a_{1},\dots, a_{N+1})\big|\nonumber\\
\leq &\sup_{(a_{1},\dots, a_{N+1})\in(iii,\beta)}\big|e_{N+1,2}\big(\mu_{n}, (a_{1},\dots, a_{N+1})\big)-e_{2,2}\big(\mu_{n}, (a_{1}, a_{1})\big)\big|\nonumber\\
&\hspace{0.4cm}+\sup_{a_{1}\!\in\mathbb{R}}\big|e_{2,2}\big(\mu_{n}, (a_{1}, a_{1})\big)-g_{N}(a_{1}, a_{1})\big|.
\end{align}
The second part of (\ref{situ3bleveln}), $\displaystyle\sup_{a_{1}\!\in\mathbb{R}}\big|e_{2,2}\big(\mu_{n}, (a_{1}, a_{1})\big)-g_{N}(a_{1}, a_{1})\big|$ converges to 0 as $n\rightarrow +\infty$ owing to Lemma~\ref{cauchyunif}. Then for the first part of (\ref{situ3bleveln}), we have
\begin{align}
&\sup_{(a_{1},\dots, a_{N+1})\in(iii,\beta)} \big|e_{N+1,2}\big(\mu_{n}, (a_{1},\dots, a_{N+1})\big)-e_{2,2}\big(\mu_{n}, (a_{1}, a_{1})\big)\big|\nonumber\\
&\hspace{0.6cm}=\sup_{(a_{1},\dots, a_{N+1})\in(iii,\beta)}\frac{e_{2,2}^{2}\big(\mu_{n}, (a_{1}, a_{1})\big)-e_{N+1,2}^{2}\big(\mu_{n}, (a_{1},\dots, a_{N+1})\big)}{e_{N+1,2}\big(\mu_{n}, (a_{1},\dots, a_{N+1})\big)+e_{2,2}\big(\mu_{n}, (a_{1}, a_{1})\big)}\nonumber\\
&\hspace{0.6cm}\leq\sup_{(a_{1},\dots, a_{N+1})\in(iii,\beta)}\frac{\int_{\mathbb{R}}\left|\xi-a_{1}\right|^{2}-\min_{1\le i\le N+1}\left|\xi-a_{i}\right|^{2}\mu_{n}(d\xi)}{\left\Vert X_{n}-a_{1}\right\Vert_{2}}\nonumber\\
&\hspace{0.6cm}\leq\sup_{(a_{1},\dots, a_{N+1})\in(iii,\beta)}\frac{\int_{\mathbb{R}}\big(\left|\xi-a_{1}\right|^{2}-\min_{2\le i\le N+1} \left|\xi-a_{2}\right|^{2}\big)_{+}\mu_{n}(d\xi)}{\left\Vert X_{n}-a_{1}\right\Vert_{2}}\nonumber\\
&\hspace{0.6cm}\leq\sup_{(a_{1},\dots, a_{N+1})\in(iii,\beta)}\frac{1}{\left\Vert X_{n}-a_{1}\right\Vert_{2}}\big[\sum_{i=2}^{N+1}\int_{\mathbb{R}}\big(\left|\xi-a_{1}\right|^{2}-\left|\xi-a_{i}\right|^{2}\big)_{+}\mu_{n}(d\xi)\big]\nonumber
\end{align}
Since $a_{1}<0$, $\left\Vert X_{n}-a_{1}\right\Vert_{2}=\big(\mathbb{E}X_{n}^{2}-2a_{1}\mathbb{E}X_{n}+\left|a_{1}\right|^{2}\big)^{1/2}\geq\sqrt{1+\left|a_{1}\right|^{2}}$. Therefore,
\begin{align*}
&\frac{\int_{\mathbb{R}}\big(\left|\xi-a_{1}\right|^{2}-\left|\xi-a_{i}\right|^{2}\big)_{+}\mu_{n}(d\xi)}{\left\Vert X_{n}-a_{1}\right\Vert_{2}}=\frac{\int_{\mathbb{R}}2(a_{i}-a_{1})\big(\xi-\frac{a_{i}+a_{1}}{2}\big)_{+}\mu_{n}(d\xi)}{\left\Vert X_{n}-a_{1}\right\Vert_{2}}\nonumber\\&\hspace{0.5cm}\leq\frac{4a_{i}\mathbb{E}X_{n}}{\sqrt{1+\left|a_{1}\right|^{2}}}\leq\frac{4a_{i}\mathbb{E}X_{n}}{\frac{1}{2}a_{i}}=8\,\mathbb{E}X_{n}.\nonumber
\end{align*}
for $i\!\in\{2, \dots, N+1\}$. Consequently, 
\begin{align*}
\sup_{(a_{1},\dots, a_{N+1})\in(iii,\beta)}\hskip-0.5cm \big|e_{N+1,2}\big(\mu_{n}, (a_{1},\dots, a_{N+1})\big)-e_{2,2}\big(\mu_{n}, (a_{1}, a_{1})\big)\big|\leq 8N\, \mathbb{E}X_{n}=8\,Ne^{-n^{2}/8}\longrightarrow0.
\end{align*}
This completes the proof. \end{proof}

\begin{proof}[Proof of Theorem~\ref{notcomplet}]
Let $\mu_{n}$ be the probability distribution of $X_{n}$ defined in (\ref{defcontre}). If for some $N\geq2$, $\big(\mathcal{P}_{2}(\mathbb{R}), \mathcal{Q}_{N, 2}\big)$ were complete, then there exists a probability mesure $\widetilde{\mu}$ in $\mathcal{P}_{2}(\mathbb{R})$ such that $\mathcal{Q}_{N, 2}(\mu_{n}, \widetilde{\mu})\longrightarrow0$. Then, $\mathcal{W}_{2}(\mu_{n}, \widetilde{\mu})\longrightarrow0$ by applying Proposition~\ref{d1peven}, which creates a contradiction.
\end{proof}
\begin{rem}The extension of this result to a   Hilbert or simply multidimensional setting, although likely, is not straightforward.
\end{rem}

\noindent{\bf Acknowledgement.} The authors thank both anonymous referees for their careful reading of the paper and fruitful suggestions. 
\begin{small}
\section*{Appendix: some examples of $c(d, |\cdot|_{r})$}\label{ldp}
\begin{proof}[Proof of Proposition~\ref{ldpprop}] 
\noindent $(i)$ is obvious.\\
\noindent$(ii)$ $c(2, |\cdot|_{1})=2$ is obvious (see Figure~\ref{new}). Now we prove that $c(2, |\cdot|_{r})=3$ for every $r\in(1, +\infty)$.

We choose $a_{1}=(0,1)$, $a_{2}=\big((1-2^{-r})^{\frac{1}{r}}, -\frac{1}{2}\big)$ and $a_{3}=\big(-(1-2^{-r})^{\frac{1}{r}}, -\frac{1}{2}\big)$. We will first show that $S_{|\cdot|_{r}}(0,1)\subset\bigcup_{1\leq i\leq3}\bar{B}_{|\cdot|_{r}}(a_{i},1)$. 

Let $(x,y)$ be any point on $S_{|\cdot|_{r}}(0,1)$, then $\big|x\big|^{r}+\big|y\big|^{r}=1$.
\begin{itemize}
\item If $\frac{1}{2}\leq y\leq 1$, then $(1-y)^{r}\leq y^{r}$ so that $\big|(x,y)-a_{1}\big|_{r}^{r}=\big|x\big|^{r}+(1-y)^{r}=1-y^{r}+(1-y)^{r}\leq1,$
that is, $(x,y)\!\in\bar{B}_{|\cdot|_{r}}(a_{1},1)$.
\item If $-1\leq y\leq \frac{1}{2}$ and $x\geq0$, then 
\begin{align}
&\big|(x,y)-a_{2}\big|_{r}^{r}=\big|x-(1-2^{-r})^{\frac{1}{r}}\big|^{r}+\big|y+\frac{1}{2}\big|^{r}=\big|(1-\big|y\big|^{r})^{\frac{1}{r}}-(1-2^{-r})^{\frac{1}{r}}\big|^{r}+\Big|y+\frac{1}{2}\Big|^{r}
\nonumber\\
&\hspace{0.4cm}
\leq \big|  |y|^{r}-2^{-r}\big| + \big|y+\frac{1}{2}\big|^{r},\nonumber
\end{align}
the last inequality is due to the fact that the function $u\mapsto u^{-\frac{1}{r}}$ is $\frac{1}{r}$-H\"older. 
As $r\geq1$, the function $y\mapsto\big||y|^{r}-2^{-r}\big| + \big|y+\frac{1}{2}\big|^{r}$ is convex over $[-1, \frac{1}{2}]$. Consequently, it attains its maximum either at $-1$ or at $\frac{1}{2}$. 
Hence, $\big|(x,y)-a_{2}\big|_{r}^{r}$ is upper bounded by $1$ since
\begin{align}
\text{if}&\;\; y=-1,\;\big| \big|y\big|^{r}-2^{-r}\big| + \big|y+\tfrac{1}{2}\big|^{r}=1-2^{-r}+2^{-r}=1,\nonumber\\
\text{if}&\;\; y=\tfrac{1}{2}, \;\;\;\Big| \big|y\big|^{r}-2^{-r}\Big| + \big|y+\frac{1}{2}\big|^{r}=\big|2^{-r}-2^{-r}\big|+1^{r}=1.\nonumber
\end{align}
This implies that $(x,y)\!\in\bar{B}_{|\cdot|_{r}}(a_{2},1)$. 
\item If $-1\leq y\leq \frac{1}{2}$ and $x\leq0$, then $(x,y)\!\in\bar{B}_{|\cdot|_{r}}(a_{3},1)$ by the symmetry of the unit sphere. 
\end{itemize}
Next, we will show $c(2, |\cdot|_{r})>2$ for every $1<r<+\infty$. 
Let $a_{1}$ and $a_{2}$ denote the two centers of balls on the sphere $S_{|\cdot|}(0, 1)$. Since the $\ell^{r}$-ball is centrally symmetric with respect to $(0, 0)$, we fix $a_{1}=(x, y)$ such that $x\in[(\frac{1}{2})^{\frac{1}{r}}, 1], \,y\in[0, (\frac{1}{2})^{\frac{1}{r}}]$ and $x^{r}+y^{r}=1$. 
\begin{itemize}
\item \textit{Case 1}. We choose $a_{2}$ such that $a_{2}$ is centrally symmetric to $a_{1}$ with respect to the center $(0, 0)$, i.e. $a_{2}=(-x, -y)$.  

We prove $z_{1}=(y, -x)\notin\cup_{i=1,2}\bar{B}_{|\cdot|_{r}}(a_{i},1)$ and $z_{2}=(-y, x)\notin\cup_{i=1,2}\bar{B}_{|\cdot|_{r}}(a_{i},1)$. In fact, if $y=0$, then $\big|a_{1}-z_{1}\big|_{r}=\big|a_{2}-z_{1}\big|_{r}=2>1$. If $y>0$, then 

\begin{align}
&\big|a_{1}-z_{1}\big|_{r}^{r}=\big|a_{2}-z_{1}\big|_{r}^{r}=\big|a_{1}-z_{2}\big|_{r}^{r}=\big|a_{2}-z_{2}\big|_{r}^{r}=(x+y)^{r}+(x-y)^{r}\nonumber\\
&\nonumber=\sum_{\substack{k=0 \\ k \text{ even}}}^{r}2\binom{r}{k}x^{r-k}y^{k}>2x^{r}\geq1\nonumber
\end{align}

\item \textit{Case 2}. The point $a_{2}$ is not centrally symmetric to $a_{1}$.

Let $H_{a_{1}}\coloneqq\{\eta=(\eta_{1}, \eta_{2})\!\in\mathbb{R}^{2} \text{ s.t. } x\cdot\eta_{2}=y\cdot\eta_{1}\}$, which is the straight line (with respect to the Euclidean distance) across the origin and $a_{1}$. Then between $z_{1}$ and $z_{2}$, there exists at least one point which is not in the same side of $H_{a_{1}}$ as $a_{2}$, and this point can not be covered by $\cup_{i=1,2}\bar{B}_{|\cdot|_{r}}(a_{i},1)$.
\end{itemize}

Figure~\ref{c233} illustrates that $c(2, |\cdot|_{r})=3$ when $r=3$.

\begin{figure}[htbp]
\centering
\begin{minipage}[t]{0.49\textwidth}
\centering
\includegraphics[height=4.5cm ,width=4.5cm]{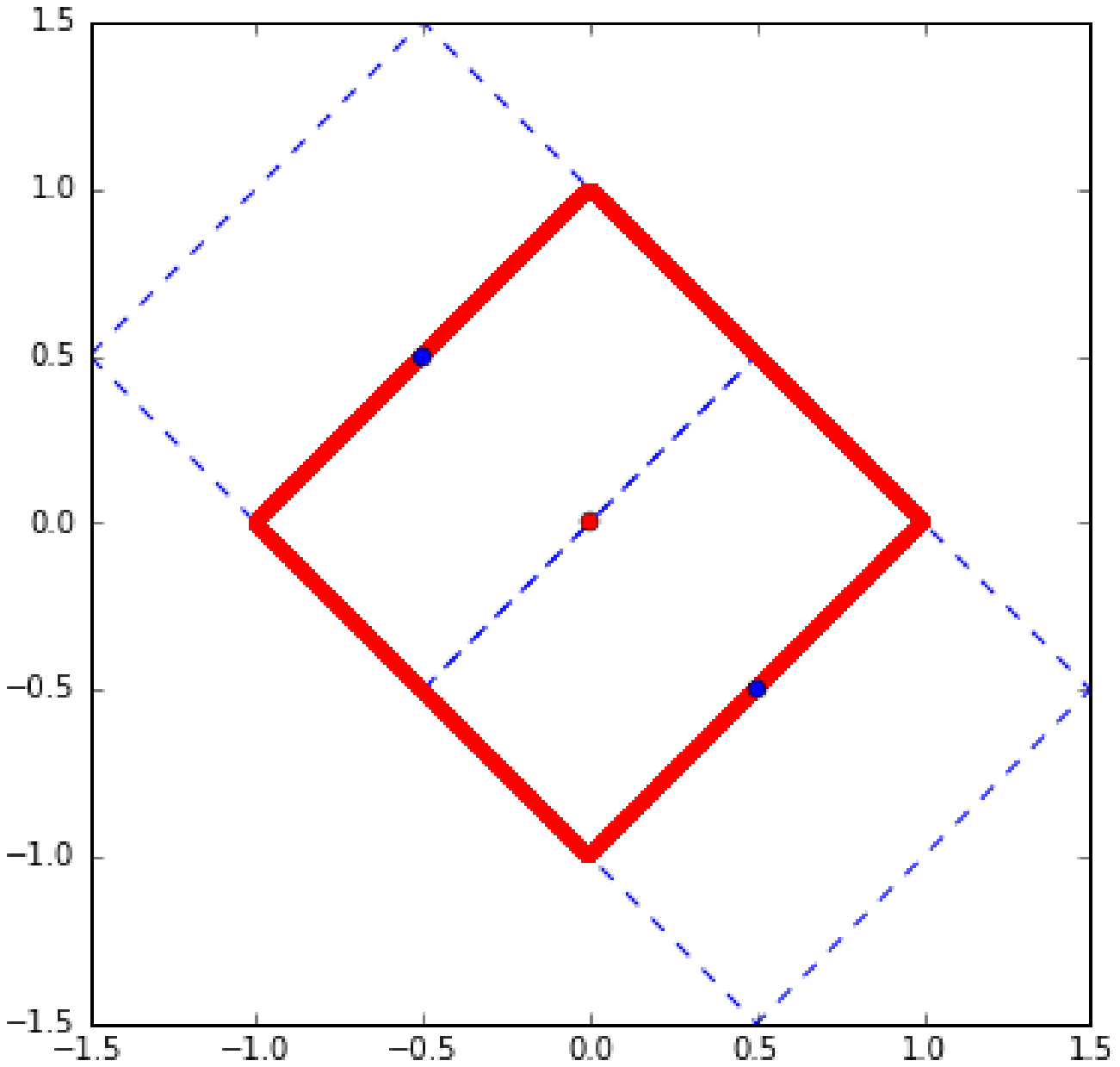}
\caption{\footnotesize $a_{1}=(-\frac{1}{2}, \frac{1}{2})$, $a_{2}=(\frac{1}{2}, -\frac{1}{2})$, \\ then $S_{|\cdot|_{1}}(0,1)\subset\bigcup_{i=1,2}\bar{B}_{|\cdot|_{1}}(a_{i},1)$}\label{new}
\end{minipage}
\begin{minipage}[t]{0.49\textwidth}
\centering
\includegraphics[height=4.5cm ,width=4.5cm]{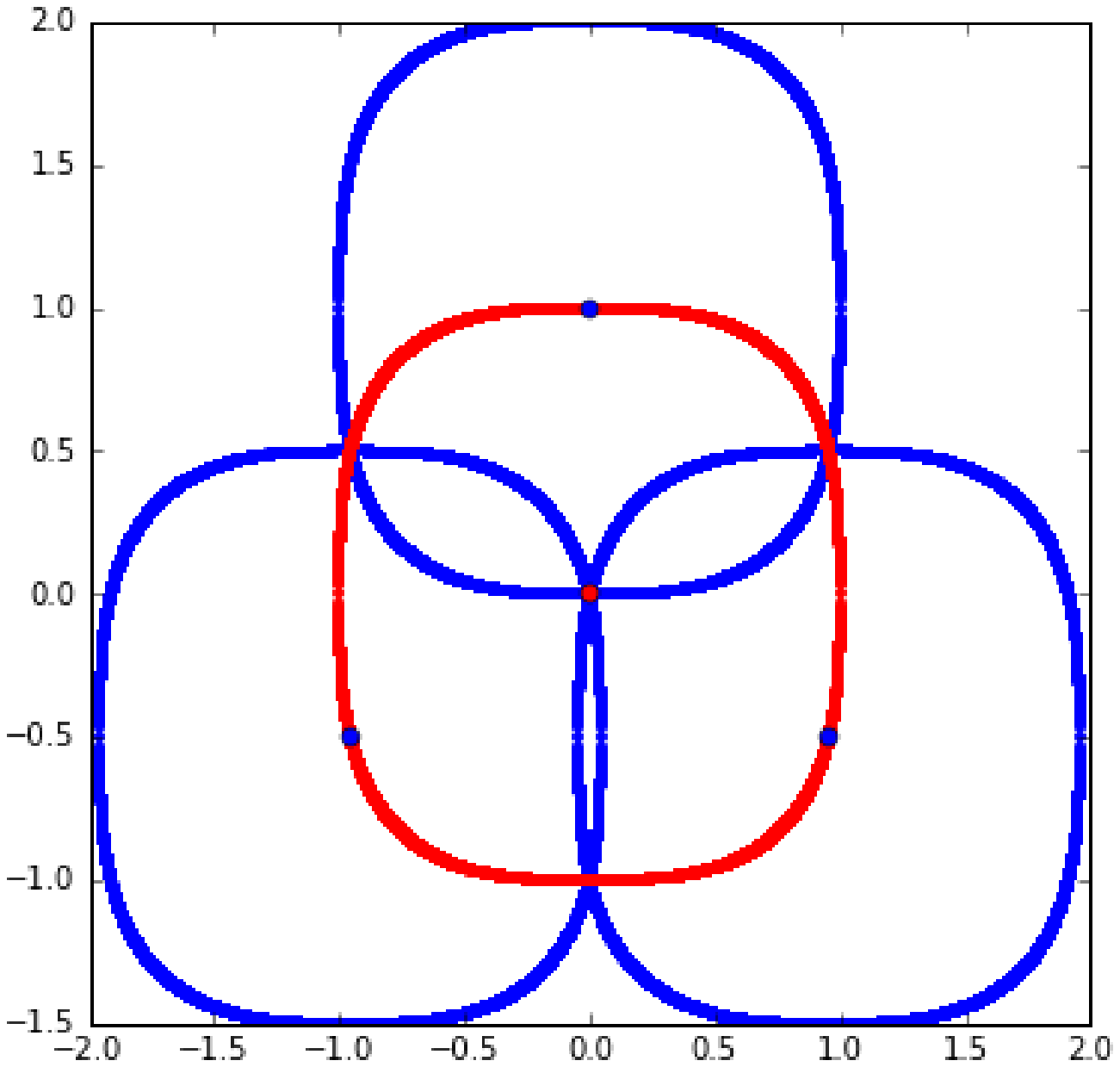}
\caption{\footnotesize $c(2, |\cdot|_{3})=3$}
\label{c233}
\end{minipage}
\end{figure}

\noindent $(iii)$ Let $a_{1}=(-1,0,\dots,0)$ and $a_{2}=(1,0,\dots,0)$. We will show that $S_{|\cdot|_{\infty}}(0,1)\subset\bigcup_{i=1,2}\bar{B}_{|\cdot|_{\infty}}(a_{i},1)$.

Let $x=(x^{1}, \dots, x^{d})\!\in S_{|\cdot|_{\infty}}(0,1)$. There exists $i_{0}$ such that $\max_{1\le i\le d}|x^{i}|\leq|x^{i_{0}}|=1$.
\begin{itemize}
\item If $i_{0}=1$, and $x^{1}=-1$, then 
$\big|x-a_{1}\big|_{\infty}=\big|x^{1}+1\big|\lor\max_{i=\{2,\dots,d\}}\big|x^{i}\big|\leq1,$
that is, $x\!\in\bar{B}_{|\cdot|_{\infty}}(a_{1},1)$.
\item If $i_{0}=1$, and $x^{1}=1$, then 
$\big|x-a_{2}\big|_{\infty}=\big|x^{1}-1\big|\lor\max_{i=\{2,\dots,d\}}\big|x^{i}\big|\leq1,$
that is, $x\!\in\bar{B}_{|\cdot|_{\infty}}(a_{2},1)$.
\item If $i_{0}\geq2$, and $x^{1}\leq0$, then 
$\big|x-a_{1}\big|_{\infty}=\big|x^{1}+1\big|\lor1\leq1,$
that is, $x\!\in\bar{B}_{|\cdot|_{\infty}}(a_{1},1)$.
\item If $i_{0}\geq2$, and $x^{1}\geq0$, then 
$\big|x-a_{2}\big|_{\infty}=\big|x^{1}-1\big|\lor1\leq1,$
that is, $x\!\in\bar{B}_{|\cdot|_{\infty}}(a_{2},1)$.
\end{itemize}
Consequently, we conclude that $S_{|\cdot|_{\infty}}(0,1)\subset\bigcup_{i=1,2}\bar{B}_{|\cdot|_{\infty}}(a_{i},1)$ and  $c(d, |\cdot|_{\infty})>1$ is obvious.

\vspace{0.2cm}
\noindent $(iv)$ 
Let $a_{i}=(0, \dots, 1, \dots, 0)$ - the $i^{th}$ coordinate of $a_{i}$ is equal to $1$ and the others equal to $0$. We will show that $S_{|\cdot|_{r}}(0,1)\subset\bigcup_{i=1}^{d}\Big(\bar{B}_{|\cdot|_{r}}(a_{i},1)\cup\bar{B}_{|\cdot|_{r}}(-a_{i},1)\Big).$

For any $x=(x^{1}, \dots, x^{d})\!\in S_{|\cdot|_{r}}(0,1)$, then there exists $i_{0}\!\in\{1, \dots, d\}$ such that $\big|x^{i_{0}}\big|\geq\frac{1}{2}$. Otherwise $\displaystyle1=\sum_{1\le i \le d}\big|x^{i}\big|^{r}< d\times2^{-r}\leq1$, which yields a contradiction.

$\bullet$  If $x^{i_{0}}\geq\frac{1}{2}$, then $\big|x-a_{i_{0}}\big|^{r}=(1-x^{i_{0}})^{r}+\sum_{i\neq i_{0}}\big|x^{i}\big|^{r}=(1-x^{i_{0}})^{r}+1-(x^{i_{0}})^{r}$.
As $x^{i_{0}}\leq\frac{1}{2}$, we have $(1-x^{i_{0}})^{r}-(x^{i_{0}})^{r}\leq0$, so that $\big|x-a_{i_{0}}\big|^{r}\leq1$, which implies that $x\!\in\bar{B}_{|\cdot|_{r}}(a_{i_{0}},1)$.

$\bullet$ If $x^{i_{0}}\leq-\frac{1}{2}$, one can similarly prove that $x\!\in\bar{B}_{|\cdot|_{r}}(-a_{i_{0}},1)$.

Consequently, we can conclude that $S_{|\cdot|_{r}}(0,1)\subset\bigcup_{i=1} ^{d}\Big(\bar{B}_{|\cdot|_{r}}(a_{i},1)\cup\bar{B}_{|\cdot|_{r}}(-a_{i},1)\Big)$.
\end{proof}
\end{small}

\bibliographystyle{abbrv}
\bibliography{quantizationcharacterization}

\end{document}